\documentclass[11pt,a4paper]{article}

\usepackage[normalem]{ulem} 
\usepackage{pifont}
\usepackage{phaistos}
\usepackage{marvosym}
\usepackage{latexsym}
\usepackage{psfrag}
\usepackage{amsmath,amssymb,cite}
\usepackage{latexsym}
\usepackage{graphicx}
\usepackage{lscape}
\usepackage{color}
\usepackage{xcolor}
\usepackage{soul}
\usepackage{afterpage}

\usepackage{hyperref}
\hypersetup{
    colorlinks=true,
    linkcolor=blue,
    filecolor=magenta,      
    urlcolor=cyan,
    citecolor=blue,
}
 \definecolor{darkgreen}{rgb}{0,.5,.5}

\DeclareMathOperator*{\argmin}{argmin}

\newcommand{\ds}{\displaystyle}
\newcommand{\nexto}{\kern -0.54em}
\newcommand{\dR}{\mathbb{R}}
\newcommand{\dN}{\mathbb{N}}

\newcommand{\proofbox}{\hspace{\fill}{$\Box$}}

\newtheorem{lemma}{Lemma}
\newtheorem{theorem}{Theorem}

\newtheorem{definition}{Definition}
\newtheorem{proposition}{Proposition}

\newtheorem{remark}{Remark}

\newtheorem{example}{Example}

\newtheorem{algorithm}{Algorithm}
\newenvironment{proof}{Proof.}{\proofbox}

\begin{document}

\author{Regina S. Burachik\footnote{Mathematics, UniSA STEM, University of South Australia, Mawson Lakes,
  S.A. 5095, Australia. 
  E-mail: regina.burachik@unisa.edu.au\,.}
\and
C. Yal{\c c}{\i}n Kaya\footnote{Mathematics, UniSA STEM, University of South Australia, Mawson Lakes,
  S.A. 5095, Australia.
  E-mail: yalcin.kaya@unisa.edu.au\,.}}

\title{\vspace{-10mm}\bf Steklov Convexification and a Trajectory Method for Global Optimization of Multivariate Quartic Polynomials}
\date{\today}
\maketitle

\vspace*{-10mm}
\begin{center}
\em Dedicated to our dear friend Marco Ant\'onio L\'opez Cerd\'a on his 70th birthday
\end{center}

\begin{abstract} 
{\sf The Steklov function $\mu_f(\cdot,t)$ is defined to average a continuous function $f$ at each point of its domain by using a window of size given by $t>0$. It has traditionally been used to approximate $f$ smoothly with small values of $t$.  In this paper, we first find a concise and useful expression for $\mu_f$ for the case when $f$ is a multivariate quartic polynomial. Then we show that, for large enough $t$, $\mu_f(\cdot,t)$ is convex; in other words, $\mu_f(\cdot,t)$ convexifies $f$.  We provide an easy-to-compute formula for $t$ with which $\mu_f$ convexifies certain classes of polynomials.  We present an algorithm which constructs, via an ODE involving $\mu_f$, a trajectory $x(t)$ emanating from the minimizer of the convexified $f$ and ending at $x(0)$, an estimate of the global minimizer of $f$.  For a family of quartic polynomials, we provide an estimate for the size of a ball that contains all its global minimizers. Finally, we illustrate the working of our method by means of numerous computational examples.}
\end{abstract}

\begin{verse} 
\end{verse}
\begin{verse}
\end{verse}

\thispagestyle{empty}

\pagestyle{myheadings}
\markboth{}{\sf\scriptsize Steklov Convexification for Global Optimization of Quartic Polynomials \ \ by R. S. Burachik and C. Y. Kaya}

\section{Introduction}

The problem of globally minimizing a {\em multivariate quartic polynomial} (MQP), $f:\dR^n\to\dR$, arises in many applications, such as signal processing \cite{Qi2003},  independent component analysis \cite{Cardoso}, blind channel equalization in digital communication \cite{Maricic}, sensor network localization \cite{KimKojWakYam2012, Nie2009, Thng1996}, hybrid system identification problems \cite{FenLagSzn2010}, and phase retrieval \cite{CaiYang2019}. Therefore, much research has been devoted to the analysis and solution methodologies for MQPs---see, e.g.   \cite{LuoZhang2010, QiJOGO, Qi2003, LiqunQi2017, WuTiaQuaUgo2014, QiWanYang2004}).  MQP optimization problems are known to be  NP-hard \cite{LuoZhang2010}, and this prompts the use of various types of relaxation techniques within a method of solution.  

In the case when it is possible to write $f$ as a {\em sum of squares} (SOS) of polynomials, one standard approach is to use a {\em semidefinite programming} (SDP) relaxation technique.  The SDP procedure has been shown to be convergent to a global minimizer of $f$, if an Archimedean condition, which implies compactness of the feasible region, holds.  The case of unbounded feasible region has also been addressed recently in \cite{JeyLasLi2014, JeyKimLeeLi2016}. However, the SDP relaxation approach is not without drawbacks: It has been observed in \cite{JeyKimLeeLi2016, LuoZhang2010} that the size of the SDP relaxation grows exponentially with the number of variables of the polynomial.  This shortcoming makes the implementation of the SDP approach difficult for large scale polynomial optimization problems. Moreover, as stated in \cite{LuoZhang2010}, no good estimate of the error is available if the process is interrupted before attaining optimality.

If a polynomial cannot be expressed in SOS form, some other solution method needs to be employed.  In \cite{Qi2003}, the authors study the global minimization problem of an even-degree multivariate polynomial whose leading degree coefficient tensor is positive definite. Such a multivariate polynomial is referred to as a {\em normal} multivariate polynomial. They give a univariate polynomial minorant of an arbitrary normal multivariate polynomial, and use it to provide an upper bound of the norm of its global minimizer.  For a subfamily of MQPs arising in signal processing, they obtain in~\cite[Theorem 5]{Qi2003} a computable upper bound on the norm of the global minimizer. In the current paper, we extend the latter result to a more general type of normal quartic polynomial---see Theorem~\ref{BL}. 

In \cite{QiJOGO}, the author shows that if a MQP $f$ has a normal quadratic essential factor, (i.e., there are quadratic polynomials $h$ and $g$, and a constant $c_0$, such that $f(x)=h(x)g(x)+c_0$), then its global minimum can either be easily found or located within the interior of the union of two given balls. 

In \cite{QiWanYang2004}, the authors look at the specific case of normal MQPs and state that the minimization of such polynomials constitutes one of the simplest cases in nonconvex global optimization. For a normal quartic polynomial, they present a criterion to find a global descent direction at a noncritical point, a saddle point, or a local maximizer. They give sufficient conditions to decide whether a local minimizer is global. They propose a global descent algorithm for finding a global minimizer of a normal quartic polynomial when $n = 2$. For $n\ge 3$, they propose an algorithm for finding  an $\varepsilon$-global minimizer. 

Another alternative approach to solving the problem of global minimization of a MQP is based on the Lipschitz properties of a general function and/or its gradient. However, these techniques require either an {\em a priori} knowledge or an estimate of the Lipschitz constants, which can be quite challenging for a multivariate function.  For more details on Lipschitzian global optimization techniques, see the recent book  \cite{SerKva2017} and the references therein. 

In the present paper, we propose a trajectory-type method for solving the problem of global optimization of MQPs. The method we propose makes use of the so-called {\em Steklov function} $\mu_f:\dR^n\times (0,\infty) \to \dR$ associated with $f$.    This type of approach was first used in~\cite{AriBurKay2019} for $n=1$, namely for univariate global optimization, where $f$ is considered to be a general coercive function as well as specifically a monic even-degree polynomial. 

The algorithm presented in~\cite{AriBurKay2019} is motivated by two essential properties of $\mu_f$: (i)~With coercive $f$, $\mu_f(\cdot,t_0)$ is convex for some large enough $t_0 > 0$, which we refer to as {\em Steklov convexification}, and (ii)~$\lim_{t\to0}\mu_f(\cdot,t) = f(\cdot)$. The algorithm itself can be summarized as follows:  Given the minimizer $x_0$ of $\mu_f(\cdot,t_0)$, the {\em trajectory} $x:(t_0,0)\to\dR^n$, which is a solution of an ODE involving $\mu_f$ with the initial condition $x(t_0) = x_0$, yields the estimate $x(0)$ of a global minimizer of $f$.  In \cite{AriBurKay2019}, the authors prove that this procedure converges to a global minimum for every {\em univariate quartic polynomial} (UQP).  In the present paper, we provide an extension of the method proposed in~\cite{AriBurKay2019} to global MQP optimization---see Theorems~\ref{thm:mu_multivar}--\ref{convexity} and Algorithm~\ref{algo1}.

The multivariate extension we provide, i.e., Algorithm~\ref{algo1}, does not necessarily yield a global optimum, as we exemplify with a counterexample in Section~\ref{counterexample}, even for the case when $n=2$.  Although a convergence proof cannot be provided for the new algorithm, extensive numerical experiments done in Section~\ref{experiments} illustrate that it can find the global minimizer (or, at least what appears to be a ``deep" local minimizer, in the absence of the knowledge of the global minimum) of many MQPs.

Trajectory based methods have surely been used in optimization before, where the trajectories employed are solutions of ODEs typically incorporating the gradient of $f(x)$, and sometimes also with additional (inertial and damping) terms.  Convergence analyses of trajectory based methods have so far been done only for finding a local minimum---see for example \cite{AttChbPeypRed2018, BotCse2018} and the references therein.  Trajectory based methods have also been proposed for global optimization, albeit without a convergence proof, to the best knowledge of the authors---see for example \cite{SnyKok2009}.

We also recall the so-called {\em backward differential flow} method, which was proposed by Zhu et al.\ in~\cite{ZhuZhaLiu2014} for the global minimization of a general differentiable function, where the trajectories are solutions of an ODE that uses the (classical) quadratic regularization.  It was illustrated in~\cite{AriBurKay2015} via a counterexample that the backward differential flow method may not yield a global minimizer, even in the case when the function is a UQP.

It should be pointed that the Steklov operator $\mu_f(\cdot,t)$ is mostly used in the literature for obtaining a smooth approximation of the function $f$ (see e.g. \cite{ErmNorWet1995, Chen2012, GarVic2013, Gupal1977, Prud2016}). Hence, most of the attention has been devoted to its properties for small values of the parameter $t>0$.  The Steklov operator is a type of {\em averaged} function introduced by Steklov \cite{Ste1907} in 1907 for studying the problem of expanding a given function into a series of eigenfunctions defined by a 2nd-order ordinary differential operator.  It was subsequently used by Kolmogorov and Fr\'echet for compactness tests in ${\cal L}^p$.

Our approach, in contrast, treats this operator as a tool for convexifying $f$ using a large enough $t$. Indeed, for certain nonconvex quartic polynomials, we show that for every $L>0$ there exists $t_0>0$ such that $\mu_f(\cdot,t_0)$ is convex over $B[0,L]$, the closed ball of center $0$ and radius $L$. Since convexification happens for $x$ in a fixed ball, we also provide an estimate, for a subfamily of quartic polynomials, of an $L$ such that all global minimizers of $f$ are contained in $B[0,L]$.  In the current paper, we list many properties of $\mu_f$ in various preliminary results and remarks.

The structure of the paper is as follows. In Section~\ref{sec:convexification}, we recall the Steklov operator and provide a concise and useful formula for the $\mu_f(\cdot,t)$ associated with any multivariate quartic polynomial. We establish conditions under which $\mu_f(\cdot,t_0)$ is convex.  We illustrate these results with examples.  Our trajectory method is motivated and Algorithm~\ref{algo1} is presented in Section~\ref{sec:trajectory}, where we show with a counterexample that convergence to a global minimum may not eventuate, even for $n=2$.  In Section~\ref{experiments} we carry out numerical experiments. We demonstrate that in many challenging cases Algorithm~\ref{algo1} does provide the global minimum of $f$.  We implement our method for a large number of randomly generated normal polynomials to extract information about the behaviour of the proposed algorithm.  We provide our concluding remarks in Section~\ref{conclusion}.

\section{Steklov Convexification}
\label{sec:convexification}

\begin{definition} \rm
The {\em Steklov} ({\em smoothing}) {\em function} (see \cite[Definition 3.8]{ErmNorWet1995}) associated with a continuous function 
$f:\dR^n \to \dR$ is denoted by $\mu_f:\dR^n\times (0,\infty) \to \dR$ and defined as
\begin{equation}  \label{Steklov}
\mu_f(x,t) := \frac{1}{(2t)^n}\,\int_{x_n-t}^{x_n+t}\cdots\int_{x_1-t}^{x_1+t} f(\tau_1,\ldots,\tau_n)\,d\tau_1\cdots d\tau_n\,,
\end{equation}
with $x:=(x_1,\ldots,x_n)\in\dR^n$.  We also refer to $\mu_f(\cdot,\cdot)$ as the {\em Steklov convexification of} $f$.
\end{definition}

\begin{remark} \rm
Since the function $f$ is continuous, $\mu_f:\dR^n\times (0,\infty)\to\dR$ is well defined and  differentiable on $\dR^n\times (0,\infty)$. If $f$ is defined on $\dR^n$, then $\mu_f$ is defined on $\dR^n\times (0,\infty)$.
\proofbox
\end{remark}

In what follows, we denote the $\ell_2$-norm by $\|\cdot\|$. We call $B[0,L]:=\{x\in \dR^n\::\: \|x\|\le L\}$ the {\em closed $\ell_2$-ball centered at $0$ and with radius $L$ in $\dR^n$}.  

In Lemma~\ref{lem:mu_univar} and Theorem~\ref{thm:mu_multivar} below, we express $\mu_f(x,t)$ in \eqref{Steklov} as a quartic polynomial function in $x$ and $t$ for the cases when $f(x)$ is a UQP and a MQP, respectively.  Lemma~1 is from~\cite{AriBurKay2019} but we repeat its short proof here in connection with \eqref{Steklov} for completeness.

\begin{lemma}[Proposition 3 in \cite{AriBurKay2019}]  \label{lem:mu_univar}
Let $f:\dR\to \dR$ be a UQP. Namely, assume that $f(x) := a_4\,x^4 + a_3\,x^3 + a_2\,x^2 + a_1\,x + a_0$, where $a_0,a_1,a_2,a_3$ and $a_4$ are real numbers. Then  \begin{equation}  \label{eq:mu_univar}
\mu_f(x,t) = f(x) + \frac{t^2}{6}\,f''(x) + \frac{a_4\,t^4}{5}= f(x) + \frac{t^2}{6}\,f''(x) + \frac{t^4}{120}f^{(4)}(x)\,.
\end{equation}
\end{lemma}
\begin{proof}
The second equality follows from the fact that $f^{(4)}(x)=24 a_4$. The first equality in \eqref{Steklov} with $n=1$ becomes
\begin{equation}  \label{Steklov_univar}
\mu_f(x,t) := \frac{1}{2t}\,\int_{x-t}^{x+t} f(\tau)\,d\tau\,.
\end{equation}
Substituting $f$ into \eqref{Steklov_univar}, integrating, expanding and rearranging, yields the first equality in \eqref{eq:mu_univar}.
\end{proof}

\begin{theorem}[Steklov Polynomial]  \label{thm:mu_multivar}
If $f:\dR^n\to \dR$ is a MQP, then $\mu_f$, defined as in \eqref{Steklov} with $n\ge1$, can be written as
\begin{equation}  \label{eq:mu_multivar}
\mu_f(x,t) = f(x) + \frac{t^2}{6}\,\sum_{i=1}^n f_{ii}(x) + \left(\frac{1}{120} \sum_{i=1}^n f_{iiii} + \frac{1}{36}\,\sum_{\substack{i,j=1 \\ j>i}}^n f_{iijj}\right) t^4\,,
\end{equation}
where $x\in\dR^n$, $t>0$, $f_{ii} := \partial^2 f/\partial x_i^2$, and $f_{iijj} := \partial^4 f/\partial x_i^2\partial x_j^2$, noting that $f_{iijj}$ are constant for all $i,j$.
\end{theorem}
\begin{proof}
We provide a proof by induction.  The basis of induction is given by $n=1$ and Lemma~\ref{lem:mu_univar}: For $n = 1$, we have
\[
\frac{1}{120} \sum_{i=1}^1 f_{iiii} + \frac{1}{36}\,\sum_{\substack{i,j=1 \\ j>i}}^1 f_{iijj}= \frac{1}{120} f^{(4)}(x)=\frac{a_4}{5}\,,
\]
because the second term on the left hand side above is zero. Using also the fact that 
\[
\frac{t^2}{6}\,\sum_{i=1}^1 f_{ii}(x)= \frac{t^2}{6}\,f^{''}(x),
\]
we see that  \eqref{eq:mu_multivar} reduces to  \eqref{eq:mu_univar}. This establishes the base case. Next, suppose that \eqref{eq:mu_multivar} holds for $n=k-1$. Namely, for every quartic polynomial $g:\dR^{k-1}\to \dR$ we have that
\begin{eqnarray} 
\mu_g(x,t) &=& g(x) + \frac{t^2}{6}\,\sum_{i=1}^{k-1} g_{ii}(x) + \left(\frac{1}{120} \sum_{i=1}^{k-1} g_{iiii} +\frac{1}{36}\,\sum_{\substack{i,j=1 \\ j>i}}^{k-1} g_{iijj}\right) t^4\,,\label{eq:mu_multivar_induct}
\end{eqnarray}
where $x\in\dR^{k-1}$, is true. What remains to do is use \eqref{eq:mu_multivar_induct} and show that \eqref{eq:mu_multivar} for $n=k$ holds. Let $f:\dR^k\to \dR$ be a MQP.  With this $f$, Equation \eqref{Steklov} can be rewritten as
\begin{equation}  \label{eq:mu_multivar_induct1}
\mu_f(x,t) = \frac{1}{2t}\int_{x_k-t}^{x_k+t}\left(\frac{1}{(2t)^{k-1}}\,\int_{x_{k-1}-t}^{x_{k-1}+t}\cdots\int_{x_1-t}^{x_1+t} f(\tau_1,\ldots,\tau_{k-1},\tau_k)\,d\tau_1\cdots d\tau_{k-1}\right) d\tau_k\,.
\end{equation}
For fixed $s\in \dR$, define $g_s:\dR^{k-1}\to \dR$ as $g_s(z_1,\ldots,z_{k-1}):=f(z_1,\ldots,z_{k-1},s)$. Note that $g_s$ is a MQP defined in $\dR^{k-1}$. Denote by $\hat x:=(x_1,\ldots,x_{k-1})$ the vector consisting of the first $k-1$ coordinates of $x$. By induction hypothesis \eqref{eq:mu_multivar_induct} applied to $g_s$ we can write,
\begin{equation}\label{IH}
\begin{array}{rcl}
\displaystyle\mu_{g_s}(\hat x,t) &=& \displaystyle  g_s(\hat x) + \frac{t^2}{6}\,\sum_{i=1}^{k-1} {(g_s)}_{ii}(\hat x) + \left(\frac{1}{120} \sum_{i=1}^{k-1} {(g_s)}_{iiii} +\frac{1}{36}\,\sum_{\substack{i,j=1 \\ j>i}}^{k-1} {(g_s)}_{iijj}\right) t^4\,,
\end{array}
\end{equation}
Using \eqref{Steklov} and the definition of $g_s$, we have that
\[
\begin{array}{rcl}
\displaystyle\mu_{g_s}(\hat x,t) &=&\displaystyle\frac{1}{(2t)^{k-1}}\,\int_{x_{k-1}-t}^{x_{k-1}+t}\cdots\int_{x_1-t}^{x_1+t} g_s(\tau_1,\ldots,\tau_{k-1})\,d\tau_1\cdots d\tau_{k-1}\\
&&\\
&=&\displaystyle \frac{1}{(2t)^{k-1}}\,\int_{x_{k-1}-t}^{x_{k-1}+t}\cdots\int_{x_1-t}^{x_1+t} f(\tau_1,\ldots,\tau_{k-1},s)\,d\tau_1\cdots d\tau_{k-1}
\end{array}
\]
{Combine the above expression with} \eqref{eq:mu_multivar_induct1} and \eqref{IH} to obtain

\begin{equation}  \label{eq:mu_multivar_induct1b}
\begin{array}{rcl}
&&\hspace{-9mm}\mu_f(x,t) = \displaystyle \frac{1}{2t}\int_{x_k-t}^{x_k+t} \mu_{g_{s}}(\hat x,t)  ds\\
&&\\
&=&\displaystyle \frac{1}{2t}\int_{x_k-t}^{x_k+t} \left( g_{s}(\hat x) + \frac{t^2}{6}\,\sum_{i=1}^{k-1} {(g_{s})}_{ii}(\hat x) + \left(\frac{1}{120} \sum_{i=1}^{k-1} {(g_{s})}_{iiii} +\frac{1}{36}\,\sum_{\substack{i,j=1 \\ j>i}}^{k-1} {(g_{s})}_{iijj}\right) t^4\right)   ds\\
&&\\
&=&\displaystyle \frac{1}{2t}\int_{x_k-t}^{x_k+t} g_{s}(\hat x) ds  + \frac{t^2}{6}\,\sum_{i=1}^{k-1}\frac{1}{2t}\int_{x_k-t}^{x_k+t} {(g_{s})}_{ii}(\hat x)  ds \\
&&\\
&&\hspace{33mm} +\ \displaystyle\frac{t^4}{120} \sum_{i=1}^{k-1} \frac{1}{2t}\int_{x_k-t}^{x_k+t} {(g_{s})}_{iiii}ds +\frac{t^4}{36}\,\sum_{\substack{i,j=1 \\ j>i}}^{k-1}\frac{1}{2t}\int_{x_k-t}^{x_k+t} {(g_{s})}_{iijj}   ds\\
&&\\
&=&\displaystyle \frac{1}{2t}\int_{x_k-t}^{x_k+t} f(\hat x,s) ds  + \frac{t^2}{6}\,\sum_{i=1}^{k-1}\frac{1}{2t}\int_{x_k-t}^{x_k+t} {f}_{ii}(\hat x,s)  ds \\
&&\\
&&\hspace{32mm} +\ \displaystyle \frac{t^4}{120} \sum_{i=1}^{k-1} \frac{1}{2t}\int_{x_k-t}^{x_k+t} {f}_{iiii}(\hat x,s) ds +\frac{t^4}{36}\,\sum_{\substack{i,j=1 \\ j>i}}^{k-1}\frac{1}{2t}\int_{x_k-t}^{x_k+t} {f}_{iijj}(\hat x,s)   ds\\
&&\\
&=& \displaystyle T_1 + T_2 +T_3+T_4,
\end{array}
\end{equation}

where 
\begin{itemize}
\item[] $T_1:=\displaystyle \frac{1}{2t}\int_{x_k-t}^{x_k+t} f(\hat x,s) ds, \quad T_2:=\displaystyle\frac{t^2}{6}\,\sum_{i=1}^{k-1}\frac{1}{2t}\int_{x_k-t}^{x_k+t} {f}_{ii}(\hat x,s)  ds $, 
\item[] $T_3:= \displaystyle\frac{t^4}{120} \sum_{i=1}^{k-1} \frac{1}{2t}\int_{x_k-t}^{x_k+t} {f}_{iiii}(\hat x,s) ds $, and  $T_4:=\displaystyle\frac{t^4}{36}\,\sum_{\substack{i,j=1 \\ j>i}}^{k-1}\frac{1}{2t}\int_{x_k-t}^{x_k+t} {f}_{iijj}(\hat x,s) ds$.
\end{itemize}
The second to last equality in \eqref{eq:mu_multivar_induct1b} is obtained by replacing the dummy variable $\tau_k$ by $s$ and using the definition of $g_s$.
Since $f$ is a MQP, we have that all derivatives of order 4 are constant. This implies that
\[
\begin{array}{rcl}
T_3=\displaystyle\dfrac{t^4}{120} \sum_{i=1}^{k-1}   f_{iiii},&\hbox{ and }& T_4=\displaystyle\dfrac{t^4}{36}\,\sum_{\substack{i,j=1 \\ j>i}}^{k-1} {f}_{iijj}.
\end{array}
\]
We proceed now to compute $T_1$ and $T_2$.   Define $h:\dR\to\dR$ such that $h(s):= f(x_1,\ldots,x_{k-1},s)=f(\hat x,s)$.
The definition of $h$ yields $h(x_k)=f(x)$, $h^{''}(x_k)= f_{kk}(x)$ and $h^{(4)}(x_k)= f_{kkkk}(x)$.
\[
\begin{array}{rcl}
T_1&=&\displaystyle \frac{1}{2t}\int_{x_k-t}^{x_k+t} f(\hat x,s) ds=\displaystyle \frac{1}{2t}\int_{x_k-t}^{x_k+t} h(s) ds=\mu_h(x_k,t).\\
\end{array}
\]
Hence, by Lemma \ref{lem:mu_univar} applied to $h$ we have 
\begin{equation}\label{T1}
\begin{array}{rcl}
T_1&=&\displaystyle\mu_h(x_k,t)= h(x_k) + \frac{t^2}{6}\,h''(x_k) + \frac{t^4}{120}h^{(4)}(x_k)\\
&&\\
&=& \displaystyle f(x)+ \frac{t^2}{6}\,f_{kk}(x) + \frac{t^4}{120}f_{kkkk}.\\
\end{array}
\end{equation}
Now let us compute $T_2$. Define $\theta^i:\dR\to\dR$ such that $\theta^i(s):= f_{ii}(x_1,\ldots,x_{k-1},s)=f_{ii}(\hat x,s)$.
The definition of $\theta^i$ yields $\theta^i(x_k)=f_{ii}(x)$, $(\theta^i)^{''}(x_k)= f_{iikk}$ and 
$(\theta^i)^{(4)}(x_k)= f_{iikkkk}(x)=0$. Hence, by Lemma \ref{lem:mu_univar} applied to $\theta^i$ we have 
\begin{equation}\label{T2}
\begin{array}{rcl}
T_2&=&\displaystyle\frac{t^2}{6}\,\sum_{i=1}^{k-1} \mu_{\theta^i}(x_k,t)=\displaystyle\frac{t^2}{6}\,\sum_{i=1}^{k-1} \left({\theta^i}(x_k) + \frac{t^2}{6}\,{\theta^i}''(x_k) + \frac{t^4}{120}{\theta^i}^{(4)}(x_k)\right)\\
&&\\
&=&\displaystyle\frac{t^2}{6}\,\sum_{i=1}^{k-1} \left(f_{ii}(x) + \frac{t^2}{6}\,f_{iikk}\right)=\displaystyle\frac{t^2}{6}\,\sum_{i=1}^{k-1} f_{ii}(x) +  \frac{t^4}{36}\,\,\sum_{i=1}^{k-1}f_{iikk}.\\
\end{array}
\end{equation}
Combining these four terms we obtain,
\begin{equation}  \label{eq:mu_multivar_induct1c}
\begin{array}{rcl}
\mu_f(x,t) &=&\displaystyle T_1 + T_2 +T_3+T_4\\
&&\\
&=&\displaystyle f(x)+ \frac{t^2}{6}\,f_{kk}(x) + \frac{t^4}{120}f_{kkkk}+\frac{t^2}{6}\,\sum_{i=1}^{k-1} f_{ii}(x) +  \frac{t^4}{36}\,\,\sum_{i=1}^{k-1}f_{iikk}\\
&&\\
&&+\displaystyle\frac{t^4}{120} \sum_{i=1}^{k-1}   f_{iiii}+\frac{t^4}{36}\,\sum_{\substack{i,j=1 \\ j>i}}^{k-1} {f}_{iijj}\\
&&\\
&=& \displaystyle f(x)+ \frac{t^2}{6}\,\sum_{i=1}^{k} f_{ii}(x) + \frac{t^4}{120} \sum_{i=1}^{k}  f_{iiii} +
 \frac{t^4}{36}\,\sum_{\substack{i,j=1 \\ j>i}}^{k} {f}_{iijj},
\end{array}
\end{equation}
which is \eqref{eq:mu_multivar} with $n=k$.  This completes the proof.
\end{proof}

\begin{remark}  \rm
Let $f:\dR^n\to \dR$ be a MQP and consider $\mu_f:\dR^n\times (0,+\infty)\to \dR$ its Steklov convexification, as given in \eqref{Steklov}. From Theorem~\ref{thm:mu_multivar}, we deduce that
\begin{eqnarray}
 \nabla_x\mu_f(x,t) &=& \nabla f(x) + \frac{t^2}{6}\sum_{i=1}^n \nabla f_{ii}(x)\in \dR^n\,,  \label{mu_grad}  \\[2mm]
\nabla_{xx}\mu_f(x,t) &=& \nabla^2f(x) + \frac{t^2}{6}\sum_{i=1}^n \nabla^2 f_{ii}(x)\,, \nonumber  \\[2mm]
& =& \nabla^2f(x) + \frac{t^2}{6}\sum_{i=1}^n \nabla^2 f_{ii}\in \dR^{n\times n}\,,\label{mu_xx}\\
 \nabla_{tx}\mu_f(x,t) &=& \frac{t}{3}\sum_{i=1}^n \nabla f_{ii}(x)\in \dR^n\,, \label{mu_tx}
\end{eqnarray}
where $\nabla_x \mu_f$ is the gradient of $\mu_f$ w.r.t. the variable $x$, $\nabla f$ is the gradient of $f$, $\nabla ^2 f$ is the Hessian of $f$, and $\nabla_{xx} \mu_f$ is the Hessian of $\mu_f$ w.r.t. the variable $x$.  Furthermore, $ \nabla_{tx}\mu_f$ is the partial derivative w.r.t. the variable $t$, of the vector $ \nabla_x\mu_f$; namely
\[
 \nabla_{tx}\mu_f(x,t)= \sum_{j=1}^n   \dfrac{\partial^2 \mu_f(x,t)}{\partial t \partial x_j}  e^j,
\]
where $\{e^1,\ldots,e^n\}$ is the canonical basis in $\dR^n$. Note that $\nabla^2 f_{ii}(x)=\nabla^2 f_{ii}$ is a constant matrix.
\proofbox
\end{remark}

The following notation and definitions will be used in the sequel. Given $B\in \dR^{n\times n}$, denote by $N(B):=\{x\in \dR^n\::\: Bx=0\},$ the null space of the matrix $B$. Given $A,B$ symmetric matrices of the same size, we write $A\succ B$  if and only if $A-B$ is positive definite. We say that $A\succeq B$  if and only if $A-B$ is positive semidefinite. Given $p\in \dN$, denote by $S_p:=\{\alpha \in \dR^p\::\: \|\alpha\|_2:=\sum_{i=1}^p (\alpha_i)^2=1\}$ the unit sphere in $\dR^p$. 

\begin{remark}\label{rem:ND}\rm
Consider the matrix
\begin{equation}  \label{C}
C:=\sum_{i=1}^n \nabla^2 f_{ii}\,.
\end{equation}
Namely, $C$ is the matrix appearing in the second term of the right hand side of \eqref{mu_xx}.  We have two possibilities: either $C\succeq 0$ or  $C\not\succeq 0$. In the latter case, we claim that there is no $t_0>0$ and no $L>0$ such that $\mu_f$ is convex over $B[0,L]\times (t_0,+\infty)$.  More precisely, for every $x\in \dR^n$ there exists $\bar t>0$ such that $\nabla_{xx}\mu_f(x,\bar t)\not\succeq 0$. Indeed, let $a<0$ be an eigenvalue of $C$ and $v$ a corresponding eigenvector with $\|v\|=1$.  Using  \eqref{mu_xx} we can write for every $x\in \dR^n$,
\[
v^T \nabla_{xx}\mu_f(x,t) v= v^T \nabla^2 f(x) v + \frac{t^2}{6} v^T C v = 
v^T \nabla^2 f(x) v +\frac{a\,t^2}{6}\,.
\] 
Since the first term in the rightmost expression is constant (for a fixed $x$) and the second term is negative, there is always a value of $t$ that makes the right hand side negative. Hence, in this case it is not possible to make $\mu_f$ convex. Consequently, it only makes sense to consider the case in which $C\succeq 0$. In a similar way, we see from the expression above that if $C=0$ then $\mu_f$ is convex over  $B[0,L]\times (t_0,+\infty)$ if and only if $f$ is convex over $B[0,L]$. Again, this case is not relevant to us, because we want to consider the case in which $f$ is not convex. Altogether, the relevant case to study is when $0\not=C\succeq 0$.
\proofbox
\end{remark}

For future use, we set up here the notation related with the spectral decomposition of the matrix $C$ given in \eqref{C}. Since $C$ is symmetric, the spectral decomposition theorem ensures that $C$ has $n$ real eigenvalues, with corresponding eigenvectors forming an orthogonal basis ${\cal B}:=\{v_1,\ldots,v_n\}$. We assume that each eigenvector $v_i$ has an associated eigenvalue $\lambda_i(C)$, $i = 1,\ldots,n$, such that $\lambda_1(C)\ge \ldots \ge \lambda_n(C)$.  In view of Remark \ref{rem:ND}, we assume from now on that $\lambda_n(C)\ge 0$ (i.e., $C$ is positive semidefinite) and that $\lambda_1(C)>0$ (i.e., $C\not=0$).  Set $r:={\rm dim\,}N(C)^{\perp}\ge 1$ (i.e., $r=n$ when  $C$ is nonsingular, and $r<n$ when $C$ is singular). When $r<n$, we have $N(C)={\rm span}[v_{r+1},\ldots,v_n]$. In this case, we denote as ${\cal N}:=\{v_{r+1},\ldots,v_n\}$ the orthonormal basis of $N(C)$ formed by eigenvectors of $C$.

Given a twice continuously differentiable function $f:\dR^n\to \dR$, and a set ${\cal A}:=\{w_1,\ldots,w_p\}$ of orthonormal vectors, we define the function $\varphi_{\cal A}:S_{p}\times\dR^n\to \dR$ as
\begin{equation}\label{fi}
\varphi_{\cal A}(\alpha,x):=\displaystyle \sum_{i=1}^p \alpha_i^2 w_i^T \nabla^2 f(x) w_i + 2\sum_{1\le i<j\le p }^p \alpha_i \,\alpha_j w_i^T \nabla^2 f(x) w_j. 
\end{equation}
Under the assumptions of $f$, $\varphi_{\cal A}$ is continuous.

\begin{lemma}\label{Lem:technical}
Let $f$ be a MQP and set $C=\sum_{i=1}^n \nabla^2 f_{ii}$. Fix $L>0$ and define $T:=B[0,L]\times (t_0,+\infty)$. 
\begin{itemize}
\item[(a)] Assume that $\lambda_n(C)= 0$ and take $\varphi_{\cal N}$ constructed as in \eqref{fi} for ${\cal A}={\cal N}=\{v_{r+1},\ldots,v_n\}$ the orthonormal basis of $N(C)$ formed by eigenvectors of $C$. Namely, $\varphi_{\cal N}:S_{n-r}\times\dR^n\to \dR$ is given by
\begin{eqnarray*}
\varphi_{\cal N}(\alpha,x) &=& \displaystyle \sum_{j,i=r+1}^n \alpha_i \,\alpha_j v_i^T \nabla^2 f(x) v_j  \\
	&=& \sum_{i=r+1}^n \alpha_i^2 v_i^T \nabla^2 f(x) v_i + 2\sum_{r+1\le i<j\le n }^n \alpha_i \,\alpha_j v_i^T \nabla^2 f(x) v_j\,. 
\end{eqnarray*}
Consider the following statements.
\begin{itemize}
 \item[(i)]  $\varphi_{\cal N}(\alpha,x)> 0$ for every $(\alpha,x)\in S_{n-r}\times B[0,L]$.
 \item[(ii)]  There exists $t_0>0$ such that $\mu_f$ is convex over the set $T$.
  \item[(iii)]  $\varphi_{\cal N}(\alpha,x)\ge 0$ for every $(\alpha,x)\in S_{n-r}\times B[0,L]$.
\end{itemize}
Then we have that $(i)\rightarrow(ii)\rightarrow(iii)$.

\item[(b)] Assume that $\lambda_n(C)>0$. Then there always exists $t_0>0$ such that $\mu_f$ is convex over the set $T$.
 
\end{itemize}
\end{lemma}

\begin{proof}
\noindent (a) Write an arbitrary unit vector $v\in \dR^n$ as a linear combination of the orthonormal basis ${\cal B}$, i.e., $v=\sum_{j=1}^n \alpha_i v_i$ where $\alpha\in S_n$. Note that $\alpha$ is well defined because ${\cal B}$ is an orthonormal basis. Assume that (i) holds. We will show that there exists $t_0>0$ such that the matrix $\nabla_{xx} \mu_f(x,t)$ is positive semidefinite for every  $(x,t)\in B[0,L]\times (t_0,+\infty)$. To establish the latter, it is enough to show that $v^T\nabla_{xx} \mu_f(x,t)v\ge 0$ for every unit vector $v$. From \eqref{mu_xx} we can write
\begin{equation}\label{lem:eq2}
\begin{array}{rcl}
v^T\nabla_{xx} \mu_f(x,t)v&=&  v^T \nabla^2 f(x) v +\dfrac{t^2}{6}\,v^TCv\\
&=& \displaystyle  \sum_{j=1}^n \alpha_j^2 v_j^T \nabla^2 f(x) v_j +\displaystyle 2\sum_{1\le i<j\le n }^n \alpha_i \,\alpha_j v_i^T \nabla^2 f(x) v_j +\dfrac{t^2}{6}\, \sum_{j=1}^r \alpha_j^2 \lambda_j(C)\\
&=& \displaystyle \varphi_{\cal B}(\alpha,x)+\dfrac{t^2}{6}\, \sum_{j=1}^r \alpha_j^2 \lambda_j(C),
\end{array}
\end {equation}
 where $\varphi_{\cal B}$ is as in \eqref{fi} for ${\cal A}={\cal B}$ and $p=n$. In the expression above, we used the fact that $v=\sum_{j=1}^n \alpha_i v_i$ in the first two terms of the second equality and the fact that ${\cal B}$ is an orthonormal basis of eigenvectors with $N(C)={\rm span}[v_{r+1},\ldots,v_n]$ in the last term of the second equality.  Assume that (ii) is not true. This means that there exist sequences $(t_k)\subset (0,+\infty)$ and $(x_k)\subset B[0,L]$ such that $(t_k)$ is strictly increasing and tending to $\infty$ and such that $\nabla_{xx} \mu_f(x_k,t_k)\not\succeq 0$. The latter means that we can find $(w_k)\subset S_n$ (i.e., unit vectors) such that
\begin{equation}\label{lem:eq3}
w_k^T \nabla_{xx} \mu_f(x_k,t_k) w_k<0,
 \end{equation}
for all $k\in \dN$.  
 By boundedness, we can extract convergent subsequences of $(w_k)$, $(\alpha^k)$, and $(x_k)$ (which we still denote as the whole sequence for simplicity), with limits $w$, $\alpha$ and $\bar x$, respectively. Use \eqref{lem:eq2} for $v=w_k$, \eqref{lem:eq3}, and the continuity of $\varphi_{\cal B}$ to obtain
 \begin{equation}\label{lem:eq31}
\begin{array}{rcl}
0\ge \lim_{k\to \infty }w_k^T\nabla_{xx} \mu_f(x_k,t_k)w_k&= & 
\lim_{k\to \infty}\varphi_{\cal B}(\alpha^k,x_k)+ \lim_{k\to \infty} \dfrac{t_k^2}{6}\, \sum_{j=1}^r ({\alpha_j^k})^2 \lambda_j(C)\\
&&\\
&=& \varphi_{\cal B}(\alpha,\bar x) +\lim_{k\to \infty} \dfrac{t_k^2}{6}\, \sum_{j=1}^r ({\alpha_j^k})^2 \lambda_j(C)\\
&&\\
&\ge& \varphi_{\cal B}(\alpha,\bar x),
\end{array}
\end {equation}
where we used the continuity of $\varphi_{\cal B}$ in the second equality and the nonnegativity of the second term in the third inequality. Note that the summation  multiplying $t_k^2$ must go to zero because $t_k$ goes to $\infty$ and $\varphi_{\cal B}$ is bounded below in $S_n\times B[0,L]$. This means that $\alpha_j^k\to 0$ for all $i\in \{1,\ldots,r\}$.  This implies that 
$w\in N(C)$, $\alpha\in S_{n-r}$ and $ \varphi_{\cal B}(\alpha,\bar x)=\varphi_{\cal N}(\alpha,\bar x)$. Using this fact and \eqref{lem:eq31} we deduce that
\[
0\ge  \lim_{k\to \infty} w_k^T \nabla_{xx} \mu_f(x_k,t_k) w_k\ge \varphi_{\cal N}(\alpha,\bar x)>0,
\] 
a contradiction. Hence, (ii) holds. Assume now that (ii) holds, let us show (iii).  For $v\in N(C)$ we have that $\alpha_j=0$ for $j\in \{1,\ldots,r\}$  and \eqref{lem:eq2} gives
\[
\begin{array}{rcl}
0\le v^T\nabla_{xx} \mu_f(x,t)v&=& \displaystyle  \sum_{j=r+1}^n \alpha_j^2 v_j^T \nabla^2 f(x) v_j +\displaystyle 2\sum_{r+1\le i<j\le n }^n \alpha_i \,\alpha_j v_i^T \nabla^2 f(x) v_j,\\
&&\\
&=&\varphi_{\cal N}(\alpha,x)
\end{array}
\]
where the first inequality follows from (ii) and the second equality from the definition of $\varphi_{\cal N}$. Since $(\alpha,x)\in S_{n-r}\times B[0,L]$ is arbitrary, the above expression yields (iii).

\noindent(b) Using \eqref{lem:eq2} with $r=n$ we obtain
\[
\begin{array}{rcl}
v^T\nabla_{xx} \mu_f(x,t)v&=&\displaystyle \varphi_{\cal B}(\alpha,x)+\dfrac{t^2}{6}\, \sum_{j=1}^r \alpha_j^2 \lambda_j(C)\\
&&\\
&\ge & \displaystyle \varphi_{\cal B}(\alpha,x)+\dfrac{t^2}{6} \lambda_n(C)\ge  \displaystyle \theta+\dfrac{t^2}{6} \lambda_n(C),
\end{array}
\]
where $\theta$ is a lower bound of $\varphi_{\cal B}$ over the compact set $T$. Since $\lambda_n(C)>0$, we can always find $t$ large enough so as to make the right hand side positive over $T$. The proof is complete.
\end{proof}

\begin{example}\rm
In some situations, condition (i) in  Lemma \ref{Lem:technical} may hold for every $x\in \dR^n$.  Consider the generalized Rosenbrock function $f:\dR^n\to \dR$ (see \cite{Goldberg,Kok}) defined as
\[
f(x)=\sum_{i=1}^{n-1}(1-x_i)^2 +100(x_{i+1}-x_{i}^2)^2. 
\]
From \cite{Kok} it can be easily checked that
$C=2400\sum_{i=1}^{n-1} e^{i,i}$, where $e^{i,j}\in \dR^{n\times n}$ is the matrix with all zeroes except at the position $(i,j)$. So $N(C)={\rm span}[e^n]$, where $e^{i}\in \dR^{n}$ is the vector with all zeroes except at the position $i$. Therefore, ${\cal N}:=\{e^n\}$. Denote by $H(x):=\nabla^2 f(x)$. From \cite[Eq. 10]{Kok}) we have that
$[H(x)](n,n)=200$ (where $A(i,j)$ denotes the position $(i,j)$ of the matrix $A$). Hence, with the notation of Lemma \ref{Lem:technical}, we have 
\[
\varphi_{\cal N}(\pm 1,x)=200,
\]  
and condition (i) in  Lemma \ref{Lem:technical} holds for every $x\in \dR^n$.
\proofbox
\end{example}

The proof of Lemma \ref{Lem:technical} is not constructive. Namely, we know when we can expect to have $\mu_f$ convex, but we don't know what the required value of $t_0$ will be. Moreover, when $\lambda_n(C)=0$, convexification over $B[0,L]\times (t_0,+\infty)$ may not be possible unless $L$ verifies the conditions of Lemma \ref{Lem:technical}(a). The next result estimates $t_0$ for a given arbitrary $L$ when $C$ is positive definite. For this, we need the following Weyl's inequality:
\begin{equation}\label{Weyl}
\lambda_n(A+B)\ge \lambda_n(A) +\lambda_n(B),
\end{equation}
for every $A,B$ symmetric. This inequality follows easily using the Rayleigh quotient. We now state and prove our convexification result for the Steklov convexification $\mu_f$. 

\begin{theorem} [Steklov Convexification] \label{convexity}
With the notation of Lemma \ref{Lem:technical}, let $f$ be a MQP and $C=\sum_{i=1}^n \nabla^2 f_{ii}$. Assume that $\lambda_n(C)>0$ and fix $L>0$. Then,  $\mu_f$ is convex over
the set $B[0,L]\times (t_0,+\infty)$, with
\[
t_0:=\sqrt{\dfrac{6\,|\theta_L|}{\lambda_n(C)}},
\]
where $\theta_L:=\displaystyle\min_{x\in B[0,L]} \lambda_n(\nabla^2 f(x))$.
\end{theorem}
\begin{proof}
Let $H_f(x):=\nabla^2f(x)$ and $H_\mu(x,t):= H_f(x)+(t^2/6)\,C$. From \eqref{mu_xx} we have that $\nabla_{xx}\mu_f(x,t)=H_\mu(x,t)$. The function $\mu_f(\cdot,t)$ is convex over the set $B[0,L]$ for all $t> t_0$ if and only if $H_\mu(x,t)\succeq 0$ for all $(x,t)\in B[0,L]\times (t_0,+\infty)$.  Our aim is to find $t_0>0$ such that the latter holds. Define $v(x):=\lambda_n(H_f(x))$. Note that $v(\cdot)$ is a continuous function of $x$, so there exists $\theta_L\in \dR$ such that $v(x)\ge \theta_L$ for every $x\in B[0,L]$. Using \eqref{Weyl} we obtain
\[
\begin{array}{rcl}
\lambda_n(H_\mu(x,t))&=&\lambda_n\left(H_f(x)+\ds\frac{t^2}{6}\,C\right)\ge \lambda_n(H_f(x)) +\dfrac{t^2}{6}\,\lambda_n(C) \\
&&\\
&=&v(x)+\dfrac{t^2}{6}\,\lambda_n(C) \ge \theta_L+\dfrac{t^2}{6}\,\lambda_n(C)\,.\\
\end{array}
\]
By the assumption on $C$, $\lambda_n(C)>0$. Hence, if $\theta_L\ge 0$ then the right hand side of the expression above is always positive and in this case $\mu_f$ is convex over $B[0,L]\times (0,+\infty)$. If $\theta_L<0$, the expression above yields 
\[
\lambda_n(H_{\mu}(x,t))\ge \theta_L+\frac{t^2\,\lambda_n(C)}{6} =-|\theta_L|+\frac{t^2\,\lambda_n(C)}{6}\,.
\]
The right-hand side is positive if 
\[
t>\sqrt{\dfrac{6\,|\theta_L|}{\lambda_n(C)}}=:t_0>0.
\]
The above expression implies that all the eigenvalues of $H_\mu(x,t)$ are positive and hence $\mu_f$ is strictly convex over the set $B[0,L]\times (t_0,+\infty)$. The proof is complete.
\end{proof}

\begin{definition}\label{def:coercive}
A continuous function $f:\dR^n\to \dR$ is {\em coercive} if $\lim_{\|x\|\to \infty} f(x)=+\infty$.
\end{definition}

\begin{remark}\rm
Theorem \ref{convexity} is useful if $L>0$ is such that $\argmin{f}\subset B[0,L]$. So that we can use the Steklov function to convexify $f$ in the region where the global minima can be found. For arbitrary MQP, the value of $L$ may not be known. As we establish later, for some types of $f$, the value of $L$ as in Theorem \ref{convexity} can be explicitly computed. If $f$ is coercive, then $L$ always exist. Thus, we restrict our analysis to the coercive case. In \cite{QiJOGO}, the value of $L$ can be explicitly found for quartic normal polynomials with a quadratic essential factor (for more details, see \cite[Proposition 14]{QiJOGO}).
\proofbox
\end{remark}

\begin{proposition}[Limiting Functions]  \label{limits}
Fix $x\in \dR^n$.  One has that
\[
\begin{array}{l}
\lim_{t\to 0}\mu_f(x,t) = f(x)\,,\ \ \lim_{t\to 0}\nabla_x \mu_f(x,t) = \nabla f(x)\,,\ \ \lim_{t\to 0}\,
\nabla_{xx}\mu_f(x,t) = \nabla^2 f(x)\,,\\ 
\\
\lim_{t\to 0}\,\nabla_{tx} \mu_f(x,t) = 0\,.
\end{array}
\]
\end{proposition}
\begin{proof}
All facts are obtained by substitution of $t=0$ into equalities \eqref{eq:mu_multivar} and \eqref{mu_grad}--\eqref{mu_tx}.
\end{proof}

\begin{definition}\label{def:vectorpol}
For $n,p$ positive integers, we say that a function $\gamma:\dR^n\to \dR^p$ is a {\em vector valued polynomial} when
\[
\gamma(x)=\sum_{i=1}^p \gamma_i(x) e^i,
\]
where $\{e^1,\ldots,e^p\}$ is the canonical basis of $\dR^p$ and $\gamma_i:\dR^n\to \dR$ is a polynomial. We say that 
$\gamma$ is a vector valued {\em linear or quadratic polynomial} when $\gamma_i$ is linear or quadratic for all $i=1,\ldots,p$. \end{definition}


The next result provides a value of $L>0$ such that $\argmin{f}\subset B[0,L]$ for a family of coercive MQPs.  It is an extension of \cite[Theorem 5]{Qi2003}. Indeed, in our theorem below 
we consider a function $f$ written as:
\begin{equation}\label{eq1:BL}
f(x)=g(x)^T\,Gg(x) + c^T h(x),
\end{equation}
where $g:\dR^n\to \dR^p$,  $h:\dR^n\to \dR^r$ are vector valued linear or quadratic polynomials, $G\in \dR^{p\times p}$ is positive definite, and $c\in \dR^r$.  In  \cite[Theorem 5]{Qi2003}, the authors assume $h(x)=g(x)$. Our proof, however, is just a slight adaptation of the one in \cite[Theorem 5]{Qi2003}, and we include it here for completeness. 

\begin{theorem}  \label{BL}
Let $f:\dR^n\to \dR$ be a MQP such that $f$ can be written as in \eqref{eq1:BL}. Denote by $\lambda$ the minimum eigenvalue of $G$. Assume further that 
\begin{itemize}
\item[(i)] There exists $L_1,R>0$ such that whenever $\|x\|>L_1$ we have
\[
\dfrac{\|h(x)\|}{\|g(x)\|}\le R.
\] 
Namely, $g$ grows (at least) as fast as $h$ for $\|x\|$ large enough.
\item[(ii)] 
\begin{itemize}
\item[(a)] If $f(0)\not=0$, define $L>L_1$ such that whenever $\|x\|>L$ we have
\[
\|g(x)\|> \max\left\{|f(0)|, \dfrac{1+\|c\| R}{\lambda}\right\}
\] 
\item[(b)] If $f(0)=0$, define $L>L_1$ such that whenever $\|x\|>L$ we have
\[
\|g(x)\|> \dfrac{\|c\| R}{\lambda}
\] 
\end{itemize}
\end{itemize}
Then, if $x^*$ is a global minimum of $f$, it satisfies $\|x^*\|\le L$. Namely, 
\[
\argmin{f}\subset B[0,L],
\]
\end{theorem}
\begin{proof}
Define $\xi:\dR^p\times \dR^r\to \dR$ as $\xi(y,z):= y^T\,Gy + c^T z$. By definition of $\lambda$ and Cauchy-Schwartz inequality we have that
\[
\xi(y,z)\ge \lambda \|y\|^2 - \|c\|\, \| z\|\,.
\]
Let $\|x\|>L_1$, with $L_1$ as in (i). The above expression, assumption (i) and the definition of $f$ yield
\[
\begin{array}{rcl}
f(x)&=&\xi(g(x),h(x))\ge \lambda \|g(x)\|^2 -\|c\|\, \|h(x)\|\\
&&\\
&=&  \|g(x)\|\left( \lambda \|g(x)\| -\|c\|\, \dfrac{\|h(x)\|}{\|g(x)\|}\right)\\
&&\\
&\ge &  \|g(x)\|\left(\lambda \|g(x)\| -R\|c\|\right)\,,
\end{array}
\]
where in the last inequality we used (i). Assume that (ii)(a) holds and take $x$ such that $\|x\|>L$. Note that the definition $L$ in (a) implies that $\|g(x)\|>|f(0)|>0$ and $(\lambda \|g(x)\| -R\|c\|)>1$.  Altogether, we find
\[
\begin{array}{rcl}
f(x)&\ge &  \|g(x)\|\left(\lambda \|g(x)\| -R\|c\|\right)> |f(0)|\ge f(0).
\end{array}
\]
Since $f(x)>f(0)$ for all $x$ such that $\|x\|>L$, we cannot have any global minimum outside $B[0,L]$. This completes the proof of case (a). Assume that (ii)(b) holds and take $x$ such that $\|x\|>L$. The definition $L$ in (b) implies that $\|g(x)\|>0$ and $(\lambda \|g(x)\| -R\|c\|)>0$. 

\[
\begin{array}{rcl}
f(x)&> &  \|g(x)\|\left( \lambda \|g(x)\| -R\|c\|\right)> 0= f(0)\,,
\end{array}
\]
and the proof follows now as in case (a).
\end{proof}

A simple situation in which conditions (i) and (ii) of Theorem \ref{BL} can be ensured is when $g$ is coercive, and grows ``more rapidly or as fast as'' $h$ when $\|x\|$ tends to infinity.


\subsection{An example normal polynomial}
To illustrate Theorem \ref{convexity}, consider a polynomial $f:\dR^n\to \dR$ defined as
\begin{equation}  \label{normal_poly}
f(x)=\sum_{i=1}^n a_i x_i^4 + x^T B x + d^T x\,,
\end{equation}
where $a = (a_1,\ldots,a_n)\in \dR^n$, $B\in \dR^{n\times n}$ is a symmetric matrix, and $d\in \dR^n$.   It is easy to check that
\[
\nabla^2 f(x)= 12 \, {\rm diag}(a_1 x_1^2,\ldots, a_n x_n^2) +2B\,,
\]
and
\[
C=\sum_{i=1}^n \nabla^2 f_{ii} = 24\,{\rm diag}(a)\,.
\]
When $\min_{i=1,\ldots,n} a_i\ge 0$ and $B\succeq 0$, the above expression and Weyl's inequality imply
that $f$ is convex. If $B\not\!\succeq 0$, then $f$ is not convex.  Indeed, in this case we have $\nabla^2 f(0)\not\!\succeq 0$.  So the relevant case arises when $B\not\!\succeq 0$.  Satisfying the conditions of the theorem means that $\min_{i=1,\ldots,n} a_i>0$. Thus, $f$ is an example of a {\em normal} quartic polynomial, as defined in \cite{QiWanYang2004}.  Formula \eqref{mu_xx} in this case becomes
\[
\nabla_{xx}\mu_f(x,t)=12 \,{\rm diag}(a_1 x_1^2,\ldots, a_n x_n^2) + 2B + 4\,t^2 {\rm diag}(a)\,.
\]
Let $a_k:=\min_{i=1,\ldots,n} a_i>0$, so $\lambda_n(C)=24\,a_k$. Since $B\not\!\succeq 0$, $\lambda_n(B)<0$. As in the proof of Theorem~\ref{BL}, we can write
\[
\begin{array}{rcl}
\lambda_n(\nabla_{xx}\mu_f(x,t))&\ge& 12\displaystyle \min_{i=1,\ldots,n} a_i x_i^2 +2 \lambda_n(B) +4 t^2 a_k\\
&&\\
  & \ge  &   2 \lambda_n(B) + 4 t^2 a_k\,,
  \end{array}
\]
which is positive as long as
\begin{equation}  \label{normal_t0}
t>t_0:=\sqrt{\dfrac{|\lambda_n(B)|}{2\,a_k}}\,.
\end{equation}
In this case, convexification of \eqref{normal_poly} is achieved for all $t$ in \eqref{normal_t0} in the whole space, because $t_0$ does not depend on $L$.  Observe that we can write
\[
f(x)=g(x)^TG g(x) + c^T h(x)\,,
\]
where $g(x):=\sum_{i=1}^n x_i^2 e^i$, $G={\rm diag}(a)$, $c=1$, and $h(x):=x^T B x +d^T x$. Hence, $f$ is of the form \eqref{eq1:BL}, and $\lambda:=a_k$.  Since $f(0)=0$, it is enough to check that conditions (i) and (ii)(b) in Theorem \ref{BL} hold for this $f$. If we can prove that (i) holds for some $L_1>0$, then (ii)(b) will follow from the fact that $g$ is coercive. Hence, let us check condition (i). Denote as $\rho(B):=\max\{|\lambda_n(B)|, |\lambda_1(B)|\}$, the spectral radius of $B$.  We can write
\[
\dfrac{\|h(x)\|}{\|g(x)\|}=\dfrac{ \|x^T B x + d^T x\|}{\|g(x)\|}\le \rho(B) \dfrac{ \|x\|^2 }{\|g(x)\|}+ \|d\|_{1}\, \dfrac{ \|x\|_{\infty}}{\|g(x)\|}.
\]
We will show that the second term on the right hand side tends to zero when $\|x\|$ tends to infinity, and that the first term remains bounded. Indeed, for every $\varepsilon>0$ take $\|x\|_{\infty}>\|d\|_1\,/\varepsilon$. Since $\|x\|=\|x\|_2\le \sqrt{n}\|x\|_{\infty}$ and $\|g(x)\|=\sqrt{\sum_{i} x_i^4}\ge \|x\|_{\infty}^2$ we have
\[
 \|d\|_1\, \dfrac{ \|x\|_{\infty}}{\|g(x)\|}\le  \|d\|_1\, \dfrac{\|x\|_{\infty}}{\|x\|_{\infty}^2}=\dfrac{\|d\|_1}{\|x\|_{\infty}}<\varepsilon,
\]
by our choice of $x$. Hence, the second term tends to zero as claimed. Now let us consider the first term. Using the same facts we arrive at
\[
\rho(B) \dfrac{ \|x\|^2 }{\|g(x)\|} \le\rho(B) \, \dfrac{{n}\|x\|_{\infty}^2}{\|x\|_{\infty}^2} = n \rho(B),
\]
for every $x\in \dR^n$. Altogether, there exists $L_1:=\dfrac{\|d\|_1}{\varepsilon}>0$ such that
\[
\dfrac{\|h(x)\|}{\|g(x)\|}\le n\,\rho(B) +\varepsilon,
\]
for $\|x\|_{\infty}>L_1$. This shows that condition (i) holds, with $R:=n\,\rho(B) +\varepsilon$. Recall that $f(0)=0$, $c=1$, and $\lambda:=a_k$. By part (b) of the theorem we need to find $L$ such that 
\[
\|g(x)\|>\dfrac{n\,\rho(B) +\varepsilon}{a_k}.
\]
We can write
\[
\|g(x)\|\ge \|x\|_{\infty}^2 > \dfrac{n\,\rho(B) +\varepsilon}{a_k}, 
\]
which holds if $\|x\|_{\infty} > \sqrt{\dfrac{n\,\rho(B) +\varepsilon}{a_k}}$. Altogether, we need to have
\[
\|x\|_{\infty}>\max\left\{\dfrac{\|d\|_1}{\varepsilon}, \sqrt{\dfrac{n\,\rho(B) +\varepsilon}{a_k}}\right\}.
\]
Hence, for 
\[
L(\varepsilon):=\max\left\{\dfrac{\|d\|_1}{\varepsilon}, \sqrt{\dfrac{n(\rho(B))+\varepsilon}{a_k}}\right\},
\]
we have $\argmin{f}\subset B_{\infty}[0,L(\varepsilon)]=\{x\::\: \|x\|_{\infty}\le L(\varepsilon)\}$. It can be checked that $L(\varepsilon)$ has a unique positive minimizer $\hat\varepsilon$.  We deduce that 
\[
\argmin\,f\subset B_{\infty}(0,L(\hat\varepsilon)).
\]


\section{A Trajectory Method Using Steklov Convexification}
\label{sec:trajectory}

The trajectory approach we formulate is based on constructing a continuously differentiable path through points where
\begin{equation}  \label{grad}
\nabla_x \mu_f(x,t) = 0\,,\quad \forall t\in(0,t_0]\,.
\end{equation}
We interpret the variable $x$ as a function dependent on $t$, i.e., $x : [0,t_0] \to \dR^n$, mapping $t \mapsto x(t)$. Assuming that all involved functions are as differentiable as needed, take the total derivative of both sides of \eqref{grad} w.r.t. the independent variable $t$, to obtain
\begin{equation}  \label{grad1}
\nabla_{xx}\mu_f(x(t),t)\,\dot{x}(t) + \nabla_{tx}\mu_{f}(x(t),t) = 0\,,\hbox{ for a.e.\ } t\in(0,t_0]\,,
\end{equation}
where $\dot{x} := (\dot{x}_1,\ldots,\dot{x}_n)$ stands for $dx/dt := (dx_1/dt,\ldots, dx_n/dt)$. In particular, we note that, for $(x(t_0),t_0) := (x_0,t_0)$, we have by \eqref{grad} that $\nabla_x \mu_f(x_0,t_0) = 0$.  After re-arranging \eqref{grad1}, one obtains the initial value problem
\begin{equation}  \label{ODE_valley1}
\dot{x}(t) = -[\nabla_{xx}\mu_f(x(t),t)]^{-1}\nabla_{tx} \mu_f(x(t),t)\,,\quad\mbox{ for a.e.\ } t\in(0,t_0]\,,\quad\mbox{with } x(t_0) = x_0\,,
\end{equation}
provided that the matrix $\nabla_{xx}\mu_f(x(t),t)$ is nonsingular for a.e.\ $t$ in $(0,t_0]$.

\begin{remark} \rm 
Suppose that $x(\cdot)$ is a solution of the ODE in~\eqref{ODE_valley1} and that $\dot{x}$ is continuous at $t=0$.  Then  Proposition~\ref{limits} implies that $\dot{x}(0)\in N(\nabla^2 f(x(0)))$, where
$N(\nabla^2 f(x(0)))$ denotes the null space of $\nabla^2 f(x(0))$. Consequently, if $\nabla^2 f(x(0))$ is nonsingular, then $\lim_{t\to 0^+}\dot{x}(t) = \dot{x}(0)=0$. 
\proofbox
\end{remark}

\subsection{An algorithm for global optimization of quartic polynomials}

Let $x_0\in \dR^n$ and $t_0>0$ be such that $\nabla_x \mu_f(x_0,t_0) = 0$ and that $\nabla_{xx}\mu_f(\cdot,t_0)$ is positive definite.  In other words, given $t_0>0$, first, by using \eqref{mu_grad}, we need to solve the following system of equations for $x_0$\,:
\begin{equation} \label{mu_grad_eqns}
\nabla f(x_0) + \frac{t_0^2}{6}\left(\sum_{i=1}^n \nabla f_{ii}\right)(x_0) = 0\,.
\end{equation}
Then, with these $t_0$ and $x_0$, using \eqref{mu_xx}--\eqref{mu_tx} in the IVP~\eqref{ODE_valley1}, we obtain
\begin{eqnarray}
&& \hspace{-1cm}\dot{x}(t) = -\frac{t}{3}\left[\nabla^2 f(x(t)) + \frac{t^2}{6}\left(\sum_{i=1}^n \nabla^2 f_{ii}\right)\right]^{-1} \left(\sum_{i=1}^n \nabla f_{ii}\right)(x(t))=:\Psi(x(t),t)\,,  \nonumber \\[2mm]
&& \hspace*{75mm}\mbox{ for a.e.\ } t\in(0,t_0]\,,\ x(t_0) = x_0\,.  \label{ODE_valley2} 
\end{eqnarray}

We denote the right-hand side in \eqref{ODE_valley2} by $\Psi:\dR^n\times \dR\to \dR^n$ for conciseness.

Algorithm~\ref{algo1} below serves to estimate a global minimizer of a MQP, $f$.

\begin{algorithm}  \label{algo1} \
\begin{description}
\vspace*{-3mm}
\item[Step \boldmath{$1$}] Choose the parameter $t_0>0$ large enough so that $\mu_f(\cdot,t_0)$ is convex. Find the (global) minimizer $x_0$ of $\mu_f(\cdot,t_0)$, i.e., solve \eqref{mu_grad_eqns} for $x_0$.
\item[Step \boldmath{$2$}] Solve the initial value problem in~\eqref{ODE_valley2}.
\item[Step \boldmath{$3$}] Report $\lim_{t\to 0^+} x(t)=:x^*$ as an estimate of a global minimizer of $f$.
\end{description}
\end{algorithm}

Algorithm~\ref{algo1} is said to be {\em well-defined} for the function $f$ if there exist $x_0$ and $t_0>0$ such that Steps~1--3 of the algorithm can be carried out. This entails, in particular,  that the solution of the IVP in Step~2 is obtained uniquely. Theorem~\ref{convexity} establishes assumptions on $f$ under which Step~1 can be carried out. The next result uses \cite[Theorem 7.1.1]{BC}, and ensures existence and uniqueness of system \eqref{ODE_valley2}. 

\begin{lemma}
Consider the function $\Psi(x,t)=(\Psi_1(x,t),\ldots,\Psi_n(x,t))$ as in \eqref{ODE_valley2}. Assume that for all $i,j=1,\ldots,n$, all the functions $\Psi_i$ and $\partial \Psi_i/\partial x_j$ are continuous (w.r.t. both $x$ and $t$) on a box $B_0\subset \dR^n\times \dR$, with $(x_0,t_0)\in B_0$. Then, the system \eqref{ODE_valley2} has a unique solution defined in the box $B_0$. \end{lemma}

\begin{remark} \rm
For the example normal polynomial in \eqref{normal_poly}, one gets, using \eqref{mu_grad}--\eqref{mu_tx},
\begin{eqnarray}
 \nabla_x\mu_f(x(t),t) &=& 4\left[\begin{array}{c}
 a_1\left(x_1(t)^3 + t^2\,x_1(t)\right) \\ \nonumber
 \vdots \\
 a_n\left(x_n(t)^3 + t^2\,x_n(t)\right) \end{array} \right] + 2\,B\,x(t) + d\,,  \label{mu_grad_normal}  \\[3mm]
\nabla_{xx}\mu_f(x(t),t) &=& 4\ {\rm diag}\left[a_1\left(3\,x_1(t)^2 + t^2\right), \ldots,  a_n\left(3\,x_n(t)^2 + t^2\right)\right] + 2\,B\,, \label{mu_xx_normal} \\[3mm]
 \nabla_{tx}\mu_f(x(t),t) &=& 8\,t
 \left[\begin{array}{c}
 a_1\,x_1(t) \\
 \vdots \\
 a_n\,x_n(t) \end{array} \right]\,, \label{mu_tx_normal}
\end{eqnarray}
The expressions~\eqref{mu_xx_normal}--\eqref{mu_tx_normal} can be substituted into \eqref{ODE_valley1} to derive the specific ODE for the normal polynomial in~\eqref{normal_poly}.
\proofbox
\end{remark}

\subsubsection{Counterexamples and comments on convergence}
\label{counterexample}

While the trajectory generated by Algorithm~\ref{algo1} has been proved to be convergent for $n=1$ in~\cite[Theorem~3]{AriBurKay2019}, a convergence proof cannot be provided for $n\ge2$.  Here we provide a numerical counterexample for $n=2$:  Consider~\eqref{normal_poly} with
{\small
\[
a = \left[\begin{array}{c} 1.05 \\ 1.96 \end{array} \right],\quad 
B = \left[\begin{array}{cc} -0.670 & -0.442 \\ -0.442 & -0.436 \end{array} \right],\quad 
d = \left[\begin{array}{r} 0.08911 \\ -0.2315\ \, \end{array} \right].
\]}
It is an easy matter to show by using \eqref{normal_t0} that $t_0 =0.694$ convexifies $f$.  When Algorithm~\ref{algo1} is invoked with this $t_0$, Step~1 of the algorithm can be carried out easily and the minimizer $x_0$ of the convex function $\mu_f(\cdot,t_0)$ can be found.  However, in Step~2, the ODE solver generates the trajectory from $t=t_0$ to $t=0$ erroneously, since $\nabla_{xx}\mu_f(x(t),t)$ becomes near-singular for values of $t$ around $0.6271$.  

For the working of the algorithm, nonsingularity, and even positive definiteness, of \\ $\nabla_{xx}\mu_f(x(t),t)$ along the trajectory emanating from $x(t_0) = x_0$ is essential, so that the trajectory can at least end at a point which is a local minimizer. However, this alone does not seem to be sufficient in guaranteeing convergence of the trajectory to a global minimizer, as this is also illustrated with Problem~Q64 in Section~\ref{QWYexamples}.

If a given MQP is separable, i.e.,
\[
f(x) = f_1(x_1) + \ldots + f_n(x_n)\,,
\]
then Algorithm~\ref{algo1} clearly yields a global minimizer, by virtue of minimizing $f_i(x)$, $i = 1,\ldots,n$, separately/individually and by the result in~\cite[Theorem~3]{AriBurKay2019}.  This suggests that Algorithm~\ref{algo1} will more likely yield a global minimizer, if the coefficients of the cross terms, such as $x_i^2x_j^2$, $x_ix_j^2x_k$, etc., in the polynomial are relatively small.

\section{Numerical Experiments}
\label{experiments}

In this section, by means of examples, we illustrate the results we presented in the preceding sections, and test Algorithm~\ref{algo1}.  In all examples, unless otherwise stated, we use {\sc Matlab}'s ODE113 to solve the IVP in~\eqref{ODE_valley2}, using the absolute and relative tolerances of $10^{-13}$.  We perform all computations on a 2018 model MacBook Pro, with macOS Mojave (version 10.14.6), the processor 2.7 GHz Intel Core i7, and a 16-GB RAM.  We use the 2019b release of {\sc Matlab}.

\subsection{A Modified Quartic Polynomial}
\label{ex1}

The solution of the problem of global minimization of the $n$-variable quartic polynomial $\sum_{i=1}^n (x_i^2 - i)^2$, which appears in \cite{Qing2006}, can be written down easily: The minimum value is zero, with the minimizers $x_i = \pm\sqrt{i}$, $i = 1,\ldots,n$.  One can also observe that the same function has a local optimum at $x_i = 0$, $i = 1,\ldots,n$.  This example is separable, so we modify it as follows:
\begin{equation}  \label{qing_extended}
f(x) = \sum_{i=1}^n (x_i^2 - i)^2 + \sum_{\substack{i,j=1 \\ j>i}}^nb_{ij}\,x_i x_j + \,\sum_{i=1}^nd_i\,x_i\,,
\end{equation}
with the real numbers $b_{ij}\in[-0.8,-0.1]$ and $d_i\in[0.1,0.5]$, where, by adding the second and third terms on the right-hand side, the minimization obviously becomes nontrivial for our purposes.

A quick inspection of $f(x)$ in \eqref{qing_extended} reveals that it can be written in the normal form described in \eqref{normal_poly}, with $a_i = 1$ and $b_{ii} = -2i$, and the $b_{ij}$ and $d_i$ as given in \eqref{qing_extended}.  Therefore, in applying Algorithm~\ref{algo1} to \eqref{qing_extended}, the bound on $t_0$ provided in \eqref{normal_t0} can be utilized.

\subsubsection{The case when \boldmath{$n=2$}}
\label{ex1.1}

Consider the global minimization of the special case of the function in \eqref{qing_extended} with $n=2$, $b_{12}=-0.7$, $d_1=0.2$ and $d_2=0.3$\,:
\begin{equation}  \label{f1}
f(x) = (x_1^2 - 1)^2 + (x_2^2 - 2)^2 - 0.7\,x_1 x_2 + 0.2\,x_1 + 0.3\,x_2\,.
\end{equation}
The graph of $f$ is depicted in Figure~\ref{ex1:graph}(a).  As can be seen from the graph, $f$ has five stationary points, namely four local minima and one local maximum, which are listed in Table~\ref{stationary_pts}. 

\begin{table}
\begin{center}
{\footnotesize
\begin{tabular}{rrrc}
\multicolumn{1}{c}{$x_1$} & \multicolumn{1}{c}{$x_2$} & \multicolumn{1}{c}{$f(x)$} & Optimality \\[1mm] \hline
 $-1.128494496206$  &  $-1.477960288995$  & $-1.727802817222$ & loc.\ min\\
  1.088972069872  &   1.442265902284  & $-0.407971945969$ & loc.\ min \\
  0.792628798894  &  $-1.398008585572$  &  0.655061617688 & loc.\ min \\
 $-0.888779137505$  &   1.352613115554  &  1.142729255749 & loc.\ min \\
  0.044197271094  &   0.033651793151  &  5.009462397888 & loc.\ max \\ \hline
\end{tabular}}
\end{center}
\caption{\sf Example~\ref{ex1.1} -- The local minima and the local maximum of $f(x)$ in \eqref{f1}.}
\label{stationary_pts}
\end{table}

The Steklov function associated with $f$ is written below using \eqref{eq:mu_multivar} in Theorem~\ref{thm:mu_multivar}.
\[
\mu_f(x,t) = f(x) + 2\,t^2\,\left(x_1^2 + x_2^2 - 1\right) + \frac{2}{5}\,t^4\,.
\]
The relevant derivatives of the Steklov function can easily be written as follows.
\begin{eqnarray}
\nabla_x\mu_f(x,t) &=& \left[\begin{array}{c}
4\,x_1^3 + 4\,(t^2 - 1)\,x_1 - 0.7\,x_2 + 0.2  \\[2mm]
4\,x_2^3 + 4\,(t^2 - 2)\,x_2 - 0.7\,x_1+ 0.3  
\end{array}\right],  \label{mu_grad_ex1} \\[2mm]
\nabla_{xx}\mu_f(x,t) &=& \left[\begin{array}{cc}
12\,x_1^2 + 4\,(t^2 - 1) & -0.7  \\[2mm]
-0.7 & 12\,x_2^2 + 4\,(t^2 - 2)
\end{array}\right],  \label{mu_xx_ex1} \\[2mm]
\nabla_{tx}\mu_f(x,t) &=&  8\,t \left[\begin{array}{c}
x_1  \\[2mm]
x_2
\end{array}\right].  \label{mu_tx_ex1}
\end{eqnarray}
It is straightforward to show that $\nabla_{xx}\mu_f(\cdot,t)$ in \eqref{mu_xx_ex1} is positive definite for all
\[
t^2\ge t_0^2 > \frac{3 + \sqrt{9 - 4\,(2 - (0.49/2))}}{2} \approx 2.0297\,.
\]
One can also use the formula in \eqref{normal_t0}, with 
\[
B = \left[\begin{array}{cc} -2 & -0.35 \\ -0.35 & -4 \end{array} \right],
\]
and subsequently $\lambda_2(B) = -4.059481005021$, and with $a_k = 1$, to get the same bound for $t_0$.  We will take $t_0 = \sqrt{2.1}$.  So the equation $\nabla_x\mu_f(x_0,t_0) = 0$ can simply be written, from \eqref{mu_grad_ex1}, with $x_0 := (x_{0,1}, x_{0,2})$, as
\begin{eqnarray*}\nonumber
4\,x_{0,1}^3 + 4.4\,x_{0,1} - 0.7\,x_{0,2} + 0.2 &=& 0\,,  \\[1mm]
4\,x_{0,2}^3 + 0.4\,x_{0,1} - 0.7\,x_{0,1} + 0.3 &=& 0\,,  
\end{eqnarray*}
the unique solution of which is found numerically as
\[
x_0 = (-0.10500662833508,\ -0.38094363094061)\,.
\]
Finally, the IVP in~\eqref{ODE_valley2} (or simply \eqref{ODE_valley1}) can be written down for this example, using \eqref{mu_xx_ex1}--\eqref{mu_tx_ex1}, as
\begin{eqnarray}
\left[\begin{array}{c}
\dot{x}_1 \\[4mm]
\dot{x}_2 \end{array} \right] &=&
-\frac{8\,t}{16\,[3\,x_1^2 + t^2 - 1]\,[3\,x_2^2 + t^2 - 2] - 0.49}\left[\begin{array}{c}
(3\,x_2^2 + t^2 - 2)\,x_1 - 0.7\,x_2 \\[4mm]
(3\,x_1^2 + t^2 - 1)\,x_2 - 0.7\,x_1  \end{array} \right], \nonumber \\[4mm]
&&\hspace*{60mm} \mbox{a.e.\ } t\in[0,\sqrt{2.1}]\,,\ \ x(\sqrt{2.1}) = x_0\,.  \label{ODE_ex1}
\end{eqnarray}
In \eqref{ODE_ex1}, we do not show the dependence of $x_i$ on $t$ for the sake of clarity in appearance.  The solution curve of the IVP above is displayed in Figure~\ref{ex1:graph}(b).  The solution found for the global minimizer $x(0) = x^*$ and the global minimum $f(x^*)$ were correct to 12dp.

\afterpage{\clearpage}
\begin{figure}[ht]
\begin{center}
\psfrag{f}{$f(x)$}
\psfrag{x1}{$x_1$}
\psfrag{x2}{$x_2$}
\includegraphics[width=120mm]{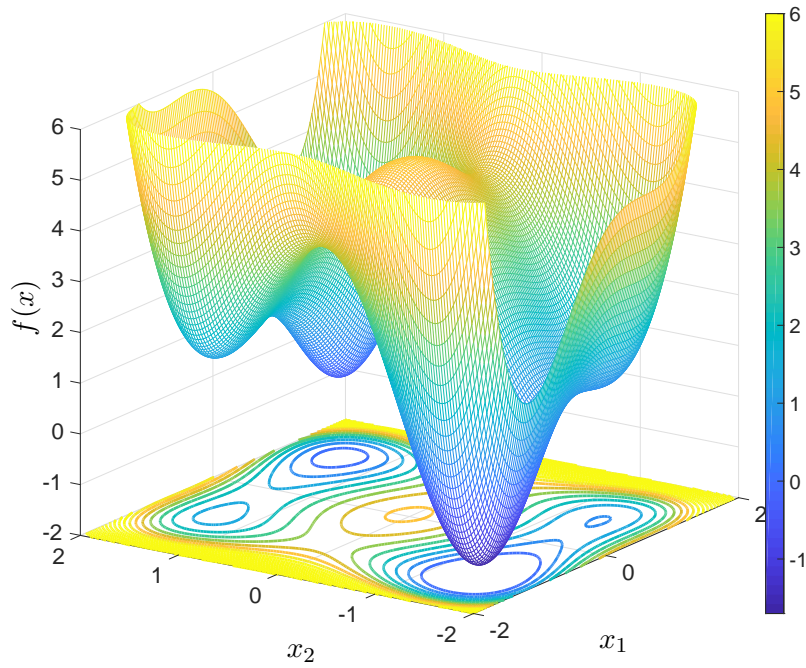} \\[3mm]
(a) The graph and the contours.
\end{center}
\begin{center}
\psfrag{f}{$f(x)$}
\psfrag{x1}{$x_1$}
\psfrag{x2}{$x_2$}
\psfrag{x0}{$x(t_0)$}
\psfrag{xf}{$x(0)$}
\includegraphics[width=120mm]{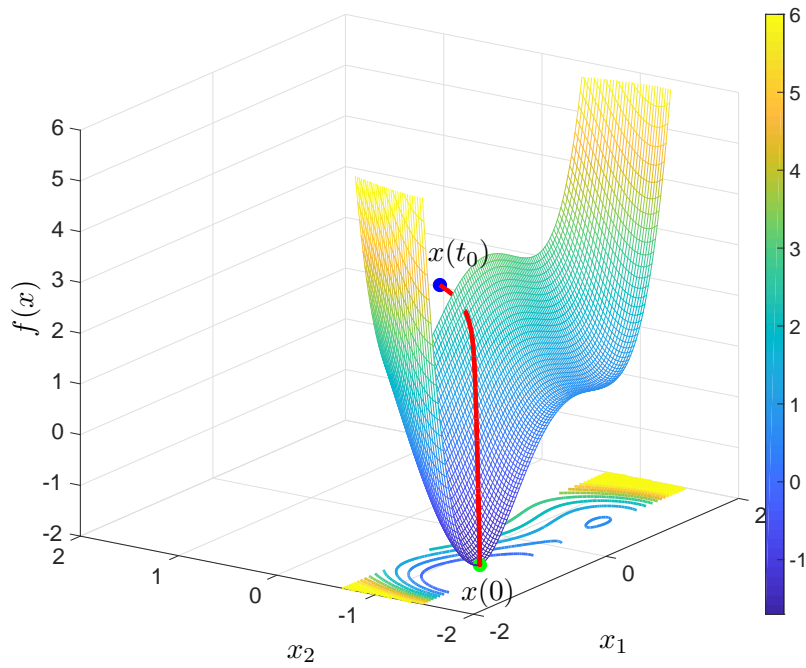} \\[3mm]
(b) A cross-sectional view revealing the trajectory.
\end{center}
\caption{\sf Example~\ref{ex1.1} -- The trajectory method using Steklov convexification for the quartic polynomial $f(x)$ in \eqref{f1}.} 
\label{ex1:graph}
\end{figure}

\subsubsection{The case when \boldmath{$n=3$}}
\label{ex1.2}

Consider global minimization of the special case of the function in \eqref{qing_extended} with $n=3$, $b_{12}=b_{13}=b_{23}=b_{31}=b_{32}=-0.7$, $d_1=d_2=d_3=0.2$\,:
\begin{equation}  \label{f2}
f(x) = (x_1^2 - 1)^2 + (x_2^2 - 2)^2 + (x_3^2 - 3)^2 - 0.7\,(x_1 x_2 + x_1 x_3 + x_2 x_3)+ 0.2\,(x_1 + x_2 + x_3)\,.
\end{equation}
\begin{table}
\begin{center}
{\footnotesize
\begin{tabular}{rrrrc}
\multicolumn{1}{c}{$x_1$} & \multicolumn{1}{c}{$x_2$} & \multicolumn{1}{c}{$x_3$} & \multicolumn{1}{c}{$f(x)$} & Optimality \\[1mm] \hline
 $-1.231880992829$  &  $-1.542141914625$  & $-1.815208552194$ & $-5.274573029462$ & loc.\ min \\
1.200891943571 & 1.520648288507 & 1.799146668019 & $-3.452471570615$ & loc.\ min \\
$-1.061403679498$ & 1.252794736408 & $-1.734795374831$ & 1.055988142832 & loc.\ min \\
1.011826545657 & $-1.292021024871$ & 1.715303007762 & 1.651815299271 & loc.\ min \\
$-1.003166837343$ & $-1.398250342156$ & 1.647600096313 & 1.720481344121 & loc.\ min \\
0.943957956849 & 1.367695525599 & $-1.669626761827$ & 2.000097716433 & loc.\ min \\
1.094178773491 & $-0.229820026863$ & 1.748687472905 & 3.474361430365 & neither \\
0.382758186693 & 0.143935301337 & $-1.724978996067$ & 5.004365968007 & neither \\
$-0.246629494866$ & $-0.103884348800$ & 1.713187106741 & 5.518166976389 & neither \\
$-1.117077740844$ & $-1.465384771470$ & 0.168916876633 & 7.589547162080 & neither \\
1.079895855992 & 1.442371632342 & $-0.131218726998$ & 8.550630493036 & neither \\
0.804174488540 & $-1.388617599449$ & 0.050802887760 & 9.810230669248 & neither \\
0.313255479037 & $-1.409450159727$ & 0.080787110272 & 9.942407427153 & neither \\
$-0.889314043601$ & $1.359163492610$ & $-0.010741630991$ & 10.007809303187 & neither \\
$-0.190950603913$ & 1.390428579020 & $-0.053353507188$ & 10.375658316925 & neither \\
$-1.009359542642$ & 0.107877288339 & 0.069364378484 & 12.878662566964 & neither \\
0.965079959377 & $-0.056352609214$ & $-0.036358449952$ & 13.219757605353 & neither \\
0.044327774003 & 0.019995236114 & 0.012915209173 & 14.007719773224 & loc.\ max \\
\hline
 \end{tabular}}
\end{center}
\caption{\sf Example~\ref{ex1.2} -- Stationary points of $f(x)$ in \eqref{f2} and their nature.}
\label{stationary_ptsb}
\end{table}
The stationary points that we could locate for $f(x)$ are listed in Table~\ref{stationary_ptsb}.  The table also indicates the nature of these points.  Writing down the derivatives of $\mu_f$ and determining $t_0$ in this case is more involved than the case when $n=2$; therefore it is convenient to put the function into the form of \eqref{normal_poly}, with
\[
a = \left[\begin{array}{c} 1 \\ 1 \\ 1 \end{array} \right],\quad B = \left[\begin{array}{ccc} -2 & -0.35 & -0.35 \\ -0.35 & -4 & -0.35 \\ -0.35 & -0.35 & -6\end{array} \right],\quad d = \left[\begin{array}{c} 0.2 \\ 0.2 \\ 0.2 \end{array} \right].
\]
Note that $\lambda_3(B) = -6.099604966650$ and so, using \eqref{normal_t0}, with $a_k = 1$,
\[
t_0 > \sqrt{\dfrac{|\lambda_3(B)|}{2}} = \sqrt{3.049802}\,.
\]
We safely take $t_0 = \sqrt{3.1}$ in Algorithm~\ref{algo1}, and (assuming that all local minimizers are those listed in Table~\ref{stationary_ptsb}) obtain the global minimizer $x(0) = x^*$ and the global minimum $f(x^*)$ correct to 12dp.  We also note that $\|\nabla f(x^*)\|_\infty = 1.6\times 10^{-13}$, which gives an idea about the accuracy of the solution, and that the hessian $\nabla^2 f(x^*)$ is positive definite, reconfirming the local minimality of the solution.

The CPU time of running a {\sc Matlab} code  implementing Algorithm~\ref{algo1} for this example can be reliably measured if the code is run 1000 times; otherwise the CPU time is too small to measure.  The CPU time of running the code 1000 times was observed to be about 5 seconds, which means that the CPU time on the average of the {\sc Matlab}  implementation of Algorithm~\ref{algo1} was 0.005 seconds.

\subsubsection{The case when \boldmath$n\ge3$}
\label{ex1.3}

We consider the polynomial \eqref{qing_extended} with various $n$, with $b_{ij}=-0.7$, $i,j =1,\ldots,n$ and $i\neq j$, and $d_i=0.2$, $i =1,\ldots,n$.  Table~\ref{ex1_performance} lists the results, including $t_0$ used by Algorithm~\ref{algo1}, the optimum value found, as well as the CPU time needed to run the {\sc Matlab} code implementing the algorithm.  

\begin{table}
\begin{center}
{\footnotesize
\begin{tabular}{rrllr}
\multicolumn{1}{c}{$n$} & \multicolumn{1}{c}{$t_0$} & \multicolumn{1}{c}{$f(x^*)$} & \multicolumn{1}{c}{$\|\nabla f(x^*)\|_\infty$} & \multicolumn{1}{c}{CPU time} \\[1mm] \hline
 3 & 1.761  & $-5.274573029462\times10^{0}$ & $1.6\times10^{-13}$ & 0.005 \\
5 & 2.354 & $-2.425189606694\times10^{1}$ & $1.4\times10^{-13}$ & 0.009 \\
10 & 3.283 & $-1.937676325137\times10^{2}$ & $7.3\times10^{-13}$ & 0.012 \\
50 & 7.196 & $-2.434927308593\times10^{4}$ & $5.3\times10^{-12}$ & 0.043 \\
100 & 10.13\,\,\, & $-1.951017166604\times10^{5}$ & $2.5\times10^{-11}$ & 0.16\hspace*{1.6mm} \\
500 & 22.50\,\,\, & $-2.442736975195\times10^{7}$ & $2.6\times10^{-10}$ & 3.9\hspace*{3.1mm} \\
1000 & 31.78\,\,\, & $-1.954665241231\times10^{8}$ & $7.3\times10^{-10}$ & 23.0\hspace*{2.9mm} \\
2000 & 44.90\,\,\, & $-1.563932649564\times10^{9}$ & $3.1\times10^{-9}$ & 120\hspace*{5.3mm} \\
5000 & 70.93\,\,\, & $-2.443840227592\times10^{10}$ & $1.2\times10^{-8}$ & 1800\hspace*{5.3mm} \\
\hline
 \end{tabular}}
\end{center}
\caption{\sf Example~\ref{ex1.3} -- Solutions for \eqref{qing_extended} with various $n$ and run times (in seconds) of Algorithm~\ref{algo1}.}
\label{ex1_performance}
\end{table}

The $\ell_\infty$-norm of the gradient of $f$, i.e., $\|\nabla f(x^*)\|_\infty$, listed in Table~\ref{ex1_performance} provides information about the accuracy of the reported solution.  We note that the hessian $\nabla^2 f(x^*)$ was checked and found to be positive definite for each $n$ in the table.  These two pieces of information reconfirm the local optimality of each solution.

{\em Local optimality} is the most we can vouch for the minimum values reported in Table~\ref{ex1_performance} (perhaps, except for $n=3$), as we have no certificate for the global optimality.  However, with growing $n$, ``deeper'' negative minimum values are obtained as expected for these kinds of polynomials.

As expected, the CPU time grows exponentially with $n$; however, given the fact that the polynomials we are dealing with are not sparse, Algorithm~\ref{algo1} might be deemed particularly successful in tackling polynomials with a large number of variables.

\subsection{Test Problems Involving Other Normal Polynomials}

\subsubsection{Test problems with \boldmath{$n=6$} from \cite{QiWanYang2004}}
\label{QWYexamples}
We consider the problems of minimizing quartic normal polynomials with six variables labelled Q61, Q62, Q63 and Q64 in \cite{QiWanYang2004}.  These problems are in the form described in \eqref{normal_poly}.  In what follows we give the descriptions of these polynomials in terms of the parameters in~\eqref{normal_poly}. \\

\noindent
{\bf Problem Q61.}
{\footnotesize
\[
a = \left[\begin{array}{c} 9 \\ 2 \\ 6 \\ 4 \\ 8 \\ 7 \end{array} \right],\quad 
B = \left[\begin{array}{cccccc} 
4 & 4 & 9 & 3 & 4 & 1 \\ 
4 & 3 & 7 & 9 & 9 & 2 \\ 
9 & 7 & 4 & 7 & 6 & 6 \\ 
3 & 9 & 7 & 4 & 2 & 6 \\
4 & 9 & 6 & 2 & 8 & 3 \\
1 & 2 & 6 & 6 & 3 & 5
\end{array} \right],\quad 
d = \left[\begin{array}{c} 2 \\ 6 \\ 5 \\ 0 \\ 0 \\ 2 \end{array} \right].
\]}
Using \eqref{normal_t0}, we get $t_0 > 1.440$.  We have set $t_0 = 1.540$ in Algorithm~\ref{algo1}, and obtained the solution
{\footnotesize
\[
x^* = [0.545218813388\ {-1.464410189792}\ {-0.720606654276}\ 1.178144265592\ 0.794065108243\ {-0.465794119448}]
\]}
\noindent
with $f(x^*) = -28.94281730403047$, $\|\nabla f(x^*)\|_\infty = 6.0\times 10^{-11}$ and $\nabla^2 f(x^*)\succ 0$.  The same solution is reported in~\cite{QiWanYang2004} with 8 dp resulting in $\|\nabla f(x^*)\|_\infty = 3.8\times 10^{-7}$. \\

\noindent
{\bf Problem Q62.}
{\footnotesize
\[
a = \left[\begin{array}{c} 4 \\ 1 \\ 8 \\ 4 \\ 6 \\ 7 \end{array} \right],\quad 
B = \left[\begin{array}{cccccc} 
4 & 0 & 0 & 3 & 0 & 3 \\ 
0 & 0 & 0 & 6 & 6 & 0 \\ 
0 & 0 & 5 & 0 & 3 & 6 \\ 
3 & 6 & 0 & 4 & 4 & 3 \\ 
0 & 6 & 3 & 4 & 4 & 5 \\ 
3 & 0 & 6 & 3 & 5 & 2
\end{array} \right],\quad 
d = \left[\begin{array}{c} 8 \\ 7 \\ 7 \\ 8 \\ 6 \\ 2 \end{array} \right].
\]}
Using \eqref{normal_t0}, we get $t_0 > 1.840$.  We have set $t_0 = 1.940$ in Algorithm~\ref{algo1}, and obtained the solution
{\footnotesize
\[
x^* = [-0.654664171603\ {-1.869516007115}\ {-0.368135071982}\ 0.819086646324\ 0.775622316964\ {-0.531322790207}]
\]}
\noindent
with $f(x^*) = -23.0056478266632$, $\|\nabla f(x^*)\|_\infty = 1.1\times 10^{-11}$ and $\nabla^2 f(x^*)\succ 0$.  The same solution is reported in~\cite{QiWanYang2004} with 8 dp resulting in $\|\nabla f(x^*)\|_\infty = 3.3\times 10^{-7}$. \\

\noindent
{\bf Problem Q63.}
{\footnotesize
\[
a = \left[\begin{array}{c} 9 \\ 7 \\ 1 \\ 4 \\ 9 \\ 9 \end{array} \right],\quad 
B = \left[\begin{array}{cccccc} 
8 & 0 & 1 & 3 & 9 & 9 \\ 
0 & 0 & 9 & 5 & 2 & 6 \\ 
1 & 9 & 4 & 1 & 1 & 8 \\ 
3 & 5 & 1 & 0 & 8 & 0 \\ 
9 & 2 & 1 & 8 & 2 & 1 \\ 
9 & 6 & 8 & 0 & 1 & 8
\end{array} \right],\quad 
d = \left[\begin{array}{c} 5 \\ 8 \\ 6 \\ 9 \\ 9 \\ 0 \end{array} \right].
\]}
Using \eqref{normal_t0}, we get $t_0 > 2.171$.  We have set $t_0 = 2.271$ in Algorithm~\ref{algo1}, and obtained the solution
{\footnotesize
\[
x^* = [-0.677847258779\ 0.915757213506\ {-1.676567471092}\ {-1.129390429402}\ 0.769478574815\ 0.740933617859]
\]}
\noindent
with $f(x^*) = -31.78036928464823$, $\|\nabla f(x^*)\|_\infty = 4.2\times 10^{-11}$ and $\nabla^2 f(x^*)\succ 0$.  The solution reported in~\cite{QiWanYang2004}, on the other hand, is another local minimizer $\hat{x}^*\neq x^*$ of $f$ with $f(\hat{x}^*) = -16.27241852 > f(x^*)$.  The same problem is also attempted in \cite{WuTiaQuaUgo2014} using a different numerical approach, resulting in the same solution as ours here but correct only up to 3 dp and with $\|\nabla f(x^*)\|_\infty = 5.7\times 10^{-3}$. \\

\noindent
{\bf Problem Q64.}
{\footnotesize
\[
a = \left[\begin{array}{c} 1 \\ 2 \\ 1 \\ 6 \\ 2 \\ 1 \end{array} \right],\quad 
B = \left[\begin{array}{cccccc} 
4 & 1 & 4 & 2 & 4 & 4 \\ 
1 & 1 & 4 & 0 & 1 & 7 \\ 
4 & 4 & 4 & 6 & 6 & 7 \\ 
2 & 0 & 6 & 6 & 7 & 9 \\ 
4 & 1 & 6 & 7 & 3 & 0 \\ 
4 & 7 & 7 & 9 & 0 & 3
\end{array} \right],\quad 
d = \left[\begin{array}{c} 8 \\ 7 \\ 6 \\ 4 \\ 7 \\ 6 \end{array} \right].
\]}
Using \eqref{normal_t0}, we get $t_0 > 2.240$.  We have set $t_0 = 2.340$ in Algorithm~\ref{algo1}, and obtained the solution
{\footnotesize
\[
x^* = [0.707423237483\ 1.239514850400\ 1.260381219594\ 1.082078205488\ {-1.644024006236}\ {-2.351712409938}]
\]}
\noindent
with $f(x^*) = -60.614291716400$, $\|\nabla f(x^*)\|_\infty = 2.1\times 10^{-10}$ and $\nabla^2 f(x^*)\succ 0$.  The solution reported in~\cite{QiWanYang2004}, on the other hand, is certainly better with
{\footnotesize
\[
\hat{x}^* = [-1.350391459\ {-1.483150332}\ {-1.369006772}\ {-1.10594118}\  1.54353024\  2.33088412]\,,
\]}
$f(\hat{x}^*) = -70.87818171 < f(x^*)$, $\|\nabla f(\hat{x}^*)\|_\infty = 2.0\times 10^{-7}$ and $\nabla^2 f(\hat{x}^*)\succ 0$.

\subsubsection{Randomly generated instances of normal polynomials with \boldmath{$n\ge2$}}
\label{rand_perf}

In order to comment further on the performance of Algorithm~\ref{algo1}, we consider randomly generated normal polynomials in the form described in~\eqref{normal_poly}.  We draw the values of the constant coefficients in~\eqref{normal_poly} at random {\em uniformly} from certain intervals such that:
\[
a_i \in [1,2]\,,\quad b_{ii} \in [-1,1]\,,\quad b_{ij} \in I_B\,,\quad d_i \in [-1,1]\,,
\]
where $i = 1,\ldots,n$, $i\neq j$, and $I_B$ is an interval emphasizing how relatively big or small the coefficients of the cross terms in~\eqref{normal_poly} will be.

\begin{table}
\begin{center}
{\footnotesize
\begin{tabular}{lrrrrrr}
\multicolumn{1}{c}{$n$ \textbackslash\ $I_B$} & \multicolumn{1}{c}{$[-0.1, 0.1]$} & \multicolumn{1}{c}{$[-0.4, 0.4]$} & \multicolumn{1}{c}{$[-0.7, 0.7]$} & \multicolumn{1}{c}{$[-1, 1]$} & \multicolumn{1}{c}{$[-2, 2]$} & \multicolumn{1}{c}{$[-10, 10]$} \\[1mm] \hline
\ \,2 & 0 (N/A) & 15 (0.02\%)  & 61 (0.06\%) & 122 (0.12\%) & 73 (0.07\%) & 4 (0.004\%) \\
\ \,5 & 12 (0.01\%) & 368 (0.4\%) & 1040 (1.0\%) & 1082 (1.1\%) & 565 (0.6\%) & 213 (0.2\%) \\
10 & 46 (0.05\%) & 2265 (2.3\%) & 3634 (3.6\%) & 3092 (3.1\%) & 1740 (1.7\%) & 1011 (1.0\%) \\
20 & 299 (0.30\%) & 9239 (9.2\%) & 8484 (8.5\%) & 6350 (6.4\%) & 4370 (4.4\%) & 3967 (4.0\%) \\[1mm]
\hline
 \end{tabular}}
\end{center}
\caption{\sf Example~\ref{rand_perf} -- Number (and percentage) of failures of Algorithm~\ref{algo1} in yielding a local minimizer $x(0)$ (which may or may not be a global minimizer), with 100,000 randomly generated normal polynomials for each $n$ and $I_B$.  The values of $b_{ij}$, $i,j = 1,\ldots,n$ and $i\neq j$, are drawn at random uniformly from the interval $I_B$, for each problem instance.}
\label{random_perf}
\end{table}

We employed \eqref{normal_t0}, in getting a convexifying $t_0$.  In finding $x_0$ in Step~1 of Algorithm~\ref{algo1}, we used (pure) Newton's method and tolerance $10^{-10}$.  In carrying out Step~2 of the algorithm, we utilized {\sc Matlab}'s solver ODE113 with the relative and absolute tolerances of $10^{-8}$.

The performance of Algorithm~\ref{algo1} as applied to a large number of random instances of the normal polynomial in~\eqref{normal_poly} is summarized in Table~\ref{random_perf}.  The size of the problems range from $n=2$ to $n=20$.  In order to get reliable statistics, we have generated 100,000 problems randomly (as described above) for each pair $(n, I_B)$.   This makes an overall 2.4 million instances.

By {\em failure} of Algorithm~\ref{algo1}, we mean that  $\|\nabla f(x(0))\|_\infty > 10^{-6}$.  In Table~\ref{random_perf}, we report for each $n$ and $I_B$ the number of failures, as well the corresponding percentage (in reference to 100,000 instances).  We have verified that $\nabla^2 f(x(0))\succ 0$ in every single one of the 2.4 million instances for which Algorithm~\ref{algo1} did not fail, furnishing the local optimality of $x(0)$.

In view of the comments made in Section~\ref{counterexample} on the convergence of the trajectory generated by Algorithm~\ref{algo1} to a global minimizer, we observe that when the coefficients of the cross terms were relatively small, i.e., when, say $I_B = [-0.1, 0.1]$, the algorithm did not fail in any of the 100,000 instances for $n=2$.  Of course, this does not mean that we probably have a proof for this case, as it is possible to find a counterexample by running an even larger number of random instances, as this was exemplified in Section~\ref{counterexample}.  For $n = 5$, 10 and 20, using the same interval $I_B$, the number of failures is rather small, at far less than $1\%$.  For the other lengths of $I_B$, the failure rates remain very small, far less than $1\%$ for $n=2$ and about $1\%$ for $n=5$.  An interesting phenomenon observed is that the failure rate starts dropping for each $n$ as $I_B$ gets significantly lengthier. Further investigation of this phenomenon is outside the scope of the current paper.

\section{Conclusion}
\label{conclusion}

A new algorithm has been proposed for global minimization of multivariate quartic polynomials.  The algorithm involves the solution of an IVP defined by the Steklov function.  We have derived new results about the properties of the Steklov function, including the convexification of a MQP $f$.  We illustrated the implementation of the algorithm and tested its performance by using a large number of numerical examples.

The method we propose is provably convergent in the special case of univariate polynomials~\cite{AriBurKay2019}, and we have numerically demonstrated in the current paper that it can fail in the multivariate case.  Therefore the current theoretical results directly ensure its convergence only in the case of separable multivariate polynomial problems. However, the problem we are dealing with is NP-complete; therefore, the question of how the method performs in the multivariate polynomial case remains an open problem.  Although the algorithm is observed to fail in finding a global minimizer for some polynomials, it is demonstrated to be successful in the great majority of the randomly generated (more than two million) instances of polynomials given in the form~\eqref{normal_poly}.  Bearing in mind the fact that there does not exist panacea even for the specific (relative simpler) form~\eqref{normal_poly}, the algorithm we propose offers a viable alternative to existing numerical approaches in the literature.

On the other hand, an analysis identifying which non-separable problems can be tackled by our method is an important and promising line of future research.  Moreover, although our study in this paper involves MQPs, one should note that the Steklov function is defined for more general (differentiable) functions.  So, it would be interesting to consider a version of the algorithm also for the global optimization of general multivariate functions.

\section*{Acknowledgments}

The authors offer their warm thanks to the anonymous reviewers for their careful reading, and the comments and suggestions they made, which in turn have improved the manuscript.  They are also grateful to the editors for efficiently handling the manuscript.

\end{document}